\documentclass[french]{smfart}
\usepackage[T1]{fontenc}
\usepackage[utf8]{inputenc}
\usepackage{graphicx}
\usepackage{amssymb}
\usepackage{amsmath}
\usepackage{amsthm}
\usepackage{color}
\usepackage{relsize}
\usepackage[all]{xy}
\usepackage{amsfonts}
\usepackage{mathrsfs}
\usepackage{latexsym}
\usepackage{geometry} 
\usepackage{textcomp}
\usepackage[english,frenchb]{babel}
\usepackage{tikz}
\usepackage[retainorgcmds]{IEEEtrantools}
\usetikzlibrary{matrix,arrows}
\usepackage{todonotes}
\usepackage{hyperref}
\usepackage[all]{hypcap}

\theoremstyle{plain}
\newtheorem*{bigtheo}{Théorème}
\newtheorem{theo}{Théorème}[subsection]
\newtheorem{prop}[theo]{Proposition}\newtheorem{conv}[theo]{Convention}
\newtheorem{lemm}[theo]{Lemme}
\newtheorem{coro}[theo]{Corollaire}
\newtheorem{hypo}[theo]{Hypothèse}
\newtheorem{ex}[theo]{Exemple}
\theoremstyle{definition}
\newtheorem{defi}[theo]{Définition}
\theoremstyle{remark}
\newtheorem{rema}[theo]{Remarque}
\newtheorem*{bigrema}{Remarque}

\newcommand{\T}{\mathcal{T}}

\newcommand\NN{\mathbb{N}} 
\newcommand\ZZ{\mathbb{Z}} 
\newcommand\QQ{\mathbb{Q}} 

\newcommand\FP{\mathbb{F}_p} 
 
\newcommand{\dans}{ \!\in\! }		

\DeclareMathOperator{\Spec}{Spec}
\DeclareMathOperator{\End}{End}
\DeclareMathOperator{\Hdg}{Hdg}
\DeclareMathOperator{\Newt}{Newt}
\DeclareMathOperator{\PR}{PR}
\DeclareMathOperator{\Gr}{Gr}
\DeclareMathOperator{\Fil}{Fil}
\DeclareMathOperator{\im}{Im}
\DeclareMathOperator{\Ker}{Ker}
\DeclareMathOperator{\Ha}{Ha}
\DeclareMathOperator{\Id}{Id}
\DeclareMathOperator{\id}{id}

\newcommand\fleche{\longrightarrow} 

\def\quotient#1#2{
  \raise0ex\hbox{$#1$}\big/\!\lower1ex\hbox{$#2$}
}

\title{Groupes $p$-divisibles avec condition de Pappas-Rapoport et invariants de Hasse
}
\author{St\'ephane Bijakowski}

\address{Imperial College,
Department of Mathematics,
180 Queen's Gate,
London SW7 2AZ,
UK}

\email{s.bijakowski@imperial.ac.uk}

\author{Valentin Hernandez}

\address{IMJ-PRG, Université Paris 6,
4 place Jussieu,
75005 Paris, 
France}

\email{valentin.hernandez@imj-prg.fr}

\begin{document}

\maketitle

\begin{abstract}
Nous étudions les groupes $p$-divisibles $G$ munis d'une action de l'anneau des entiers d'une extension finie (possiblement ramifiée) de $\mathbb{Q}_p$ sur un schéma de caractéristique $p$. Nous supposons de plus que le groupe $p$-divisible satisfait la condition de Pappas-Rapoport pour une certaine donnée $\mu$ ; cette condition consiste en une filtration sur le faisceau des différentielles $\omega_G$ satisfaisant certaines propriétés. Sur un corps parfait, nous définissons les polygones de Hodge et de Newton pour de tels groupes $p$-divisibles, en tenant compte de l'action. Nous montrons que le polygone de Newton est au-dessus du polygone de Hodge, lui-même au-dessus d'un certain polygone dépendant de la donnée $\mu$. \\
Nous construisons ensuite des invariants de Hasse pour de tels groupes $p$-divisibles sur une base arbitraire de caractéristique $p$. Nous prouvons que l'invariant de Hasse total est non nul si et seulement si le groupe $p$-divisible est $\mu$-ordinaire, c'est-à-dire si son polygone de Newton est minimal. Enfin, nous étudions les propriétés des groupes $p$-divisibles $\mu$-ordinaires. \\
La construction des invariants de Hasse s'applique en particulier aux fibres spéciales des modèles des variétés de Shimura PEL construits par Pappas et Rapoport. 
\end{abstract}

\selectlanguage{english}
\begin{abstract}
We study $p$-divisible groups $G$ endowed with an action of the ring of integers of a finite (possibly ramified) extension of $\mathbb{Q}_p$ over a scheme of characteristic $p$. We suppose moreover that the $p$-divisible group $G$ satisfies the Pappas-Rapoport condition for a certain datum $\mu$ ; this condition consists in a filtration on the sheaf of differentials $\omega_G$ satisfying certain properties. Over a perfect field, we define the Hodge and Newton polygons for such $p$-divisible groups, normalized with the action. We show that the Newton polygon lies above the Hodge polygon, itself lying above a certain polygon depending on the datum $\mu$. \\
We then construct Hasse invariants for such $p$-divisible groups over an arbitrary base scheme of characteristic $p$. We prove that the total Hasse invariant is non-zero if and only if the $p$-divisible group is $\mu$-ordinary, i.e. if its Newton polygon is minimal. Finally, we study the properties of $\mu$-ordinary $p$-divisible groups. \\
The construction of the Hasse invariants can in particular be applied to special fibers of PEL Shimura varieties models as constructed by Pappas and Rapoport.
\end{abstract}

\selectlanguage{french}

\tableofcontents

\section*{Introduction}

L'étude des congruences entre formes modulaires s'est avérée, au delà de son caractère esthétique, avoir de nombreuses conséquences importantes en théorie des nombres ; que 
ce soit l'article fondateur de Serre \cite{Se} d'où sont déduites les congruences entre les valeurs des fonctions z\^eta aux entiers négatifs, ou bien l'article de Deligne et Serre \cite{DS} qui utilise les congruences pour construire les représentations d'Artin associées aux formes modulaires de poids 1. Ces deux aspects ont depuis été grandement généralisés, et dans les cas cités précédemment ils sont basé sur l'utilisation de la série d'Eisenstein $E_{p-1}$ ($p\geq 5$ un nombre premier), puisque l'ideal qui définit (pour $SL_2(\ZZ)$) les congruences modulo $p$ entre formes modulaires est donné par $E_{p-1} - 1$.

On peut interpréter géométriquement ce qui précède, puisque $E_{p-1}$ est une section sur la courbe modulaire du fibré $\omega^{p-1}$. 
Pour chaque nombre premier où une courbe elliptique (disons sur un corps de nombres) a bonne réduction, on peut regarder le nombre de points modulo $p$ (dans $\overline{\FP}$) de sa $p$-torsion. 
Une courbe elliptique (à bonne réduction en $p$) est dite supersingulière en $p$ si sa $p$-torsion n'a qu'un point sur 
$\overline{\FP}$, et ordinaire si elle en a $p$. Si la courbe elliptique est ordinaire, alors sa $p$-torsion est une extension d'une partie multiplicative par une partie étale. De plus, on peut associer à chaque courbe elliptique un invariant modulo $p$, dit de Hasse, décrivant l'action du 
Verschiebung sur le faisceau conormal. Cet invariant est inversible exactement lorsque la courbe elliptique est ordinaire (en $p$). En fait cet invariant est la réduction modulo $p$ de la série d'Eisenstein $E_{p-1}$ ; son lieu d'annulation étant alors un fermé de la courbe modulaire. L'invariant de Hasse est donc intimement lié aux congruences entre formes modulaires.

Plus généralement, étant donné un groupe $p$-divisible, on peut lui associer des polygones de Hodge et Newton et on dit qu'il est ordinaire lorsque ces deux polygones sont égaux. Cette condition est équivalente au fait que la $p$-torsion soit l'extension d'une partie multiplicative par une partie étale. Comme dans le cas des courbes elliptiques, on peut associer à $G$ un invariant de Hasse qui détecte si le groupe est ordinaire ou non. Étant donné une variété de Shimura (PEL non ramifiée en $p$), on peut alors grâce à ce qui précède construire un ouvert, le lieu ordinaire. Lorsqu'il est non vide, cet ouvert est dense (\cite{NO}). Le polygone de Newton permet plus généralement de construire une stratification de la variété.

En général, malheureusement, fixer des endomorphismes dans la donnée de Shimura peut impliquer que le lieu ordinaire est vide. Dans ce cas on doit construire de nouveaux 
invariants. Lorsque les endomorphismes sont non ramifiés au dessus de $p$, de nombreux outils permettent d'étudier de telles familles de groupes $p$-divisibles. En particulier, on peut définir les polygones de Newton et de Hodge normalisés avec l'action (étudiés dans \cite{Kot2} et \cite{RR}), et on définit le lieu $\mu$-ordinaire comme le lieu o\`u ces deux polygones sont égaux. Ce lieu, défini par Rapoport dans une conférence en 1996, permet donc de construire un ouvert, qu'il a conjecturé être dense pour les variétés de niveau hyperspécial en $p$ (voir \cite{Ra} pour plus de détails). Ce résultat a été prouvé par Wedhorn dans \cite{Wed}. Les groupes $p$-divisibles $\mu$-ordinaires sur $\overline{\FP}$ ont ensuite été étudiés dans \cite{Moo}, et l'ouvert des variétés de Shimura précédentes dans \cite{Wed}. De plus, Goldring et Nicole (\cite{GN}) ont introduit des invariants de Hasse dans le cas des variétés de Shimura unitaires, qui détectent l'ouvert $\mu$-ordinaire. 

En fait le problème précédent est local, détaillons ce qui est connu. Soit $S$ un schéma de caractéristique $p$, et $G\fleche S$ un groupe $p$-divisible muni d'une action d'une $\ZZ_p$-algèbre $\mathcal O$, qui est l'anneau des entiers d'une extension finie non ramifiée de $\mathbb{Q}_p$ . On peut alors associer à $G$ une famille (indéxée par les points géométriques de $S$) de polygones de Newton, ainsi qu'une famille de polygones de Hodge, qui tiennent compte de l'action de $\mathcal O$. Les polygones de Hodge sont localement constant sur $S$ (voir \cite{Kot2} ou les rappels de \cite{Her1}). L'ouvert $\mu$-ordinaire est alors le lieu où les polygones de Hodge et Newton coïncident.
On peut aussi associer à $G$ un invariant de Hasse $\mu$-ordinaire, qui est une section sur $S$ d'un fibré en droites, le déterminant du faisceau conormal de $G$ (à une certaine puissance explicite) ; voir \cite{Her1} pour la construction à l'aide de la cohomologie cristalline de $G$, ou \cite{KW} pour une construction utilisant les $G$-Zip.
Ces derniers invariants ont de nombreuses applications en dehors de la géométrie des variétés de Shimura (\cite{DS} dans le cas de la courbe modulaire, puis \cite{Box,GK} dans un cadre plus général les utilisent pour construire des représentations galoisiennes par interpolation).

Si on autorise l'anneau $\mathcal{O}$ à être ramifié, le problème devient plus difficile. Il semble alors nécessaire d'introduire les modèles entiers de Pappas-Rapoport à la place des modèles entiers de Kottwitz, car ces derniers ne sont pas plats. Sur ces modèles, on dispose d'un groupe $p$-divisible $G$ satisfaisant une condition de Pappas-Rapoport, c'est-à-dire une filtration sur le faisceau $\omega_G$ dont on fixe les dimensions des gradués. Dans le cas des variétés de Hilbert, les dimensions des gradués sont égaux à $1$ ; Reduzzi et Xiao ont alors montré que l'invariant de Hasse classique s'exprimait comme un produit d'invariants partiels. Cette construction a été généralisée dans \cite{Bijpar} au cas où la dimension des gradués est constante. Si cette dernière condition n'est pas vérifiée, l'invariant de Hasse est toujours nul (et le groupe $p$-divisible ne peut pas être ordinaire). C'est ce cas que nous proposons d'étudier.

Soit $p$ un nombre premier, et $L$ une extension finie de $\mathbb{Q}_p$. On note $O_L$ son anneau des entiers, $k_L$ le corps résiduel, $f$ le degré résiduel, $e$ l'indice de ramification, et $\mathcal{T}$ l'ensemble des plongements de $L^{nr}$ dans $\overline{\mathbb{Q}_p}$, où $L^{nr}$ est l'extension maximale non ramifiée contenue dans $L$. Soient enfin  $\mu = (d_{\tau,i})_{\tau \in \mathcal{T}, 1 \leq i \leq e}$ des entiers. On étudie alors des groupes $p$-divisibles $G$ sur un schéma $S$ de caractéristique $p$, munis d'une action de $O_L$ (i.e. un morphisme $\iota : O_L \to \End(G)$), et satisfaisant une condition de Pappas-Rapoport pour la donnée $(d_\bullet)$. Cette condition consiste en une filtration $\Fil^\bullet$ sur le faisceau $\omega_G$ dont les gradués sont de dimension $d_{\tau,i}$ (voir définition \ref{defPR}). 
On note $\underline G$ le groupe avec ces données.

Nos résultats sont alors les suivants (voir théorème \ref{thr118} et théorème \ref{thr119}),

\begin{bigtheo}
Soit $\underline G = (G,\iota,\Fil^{\bullet})$ un groupe $p$-divisible avec condition de Pappas-Rapoport $\mu$ sur un corps parfait $k$ de caractéristique $p$, muni d'une action de $O_L$ et d'une donnée de Pappas-Rapoport. On peut alors associer à $\underline G$ des polygones (convexes) de Newton $\Newt(\underline G)$, de Hodge $\Hdg(\underline G)$, et définir un polygone de Pappas-Rapoport $\PR(\mu)$. Ces polygones vérifient les inégalités,
\[ \Newt(\underline G) \geq \Hdg(\underline G) \geq \PR(\mu)\]
De plus, les polygones $\Newt(\underline G)$ et $\Hdg(\underline G)$ vérifient la propriété de filtration Hodge-Newton : si ces deux polygones ont un point de contact qui est un point de rupture pour $\Newt(\underline G)$ alors il existe un unique scindage $\underline G = \underline G' \oplus \underline G''$ correspondant à ce contact.
\end{bigtheo}

\begin{bigrema} En fait le théorème vaut plus généralement pour $\underline M = (M,F,\iota,\Fil^{\bullet})$ un $F$-cristal sur $k$. 
\end{bigrema}

Dans le cas où $O_L = \ZZ_p$ (la condition de Pappas-Rapoport est alors triviale), le théorème précédent est le théorème de Mazur ainsi que le théorème de Katz 
sur la rupture Hodge-Newton (voir \cite{Katz}). Ces résultats ont été généralisés dans de nombreux cas (voir \cite{RR, MV, Her1}), mais pas à notre connaissance dans les 
cas ramifiés.

Un groupe $p$-divisible $\underline G$ sur un corps parfait $k$ de caractéristique $p$ est alors dit $\mu$-ordinaire si $\Newt(\underline G) = \PR(\mu)$. Si la donnée $\mu = (d_{\tau,i})_{\tau,i}$ est ordinaire, i.e. si le polygone PR a pour pentes $0$ ou $1$, on retrouve la condition d'ordinarité usuelle. En revanche, si ce n'est pas le cas, la condition de Pappas-Rapoport induit une obstruction à être ordinaire. 
On dit que $\underline G$ vérifie la condition de Rapoport généralisée si $\Hdg(\underline G) = \PR(\mu)$. Dans le cas d'une donnée $\mu = (d_{\tau,i})_{\tau,i}$ ordinaire, on retrouve la condition de Rapoport usuelle (i.e. le module $\omega_G$ est libre sur $k \otimes_{\mathbb{Z}_p} O_L$). \\
$ $\\
\indent La deuxième partie de cet article est consacrée à la construction d'invariants de Hasse. Pour cela la construction est purement locale, et est valable pour les groupes de Barsotti-Tate tronqués d'échelon 1 (voir quand même la remarque \ref{remBTO1}). Cela nous permet aussi de caractériser la $\mu$-ordinarité.
De plus, on a d'autres caractérisations possibles du fait d'être $\mu$-ordinaire dans le cas d'un corps algébriquement clos ; on peut construire explicitement un groupe $p$-divisible $X^{ord}$ qui est $\mu$-ordinaire (voir définition \ref{def33}) auquel tout groupe $p$-divisible $\mu$-ordinaire est isomorphe. Enfin, le fait d'être $\mu$-ordinaire ne dépend que de la $p$-torsion, généralisant ainsi les résultats de Moonen (\cite{Moo}). \\
Notre théorème est le suivant (voir proposition \ref{prop215}, théorème \ref{equimuord} et proposition \ref{pro38}).

\begin{bigtheo}
Supposons que pour tout $\tau$, $d_{\tau,1} \geq d_{\tau,2} \geq \dots \geq d_{\tau,e}$. Soit $S$ un schéma sur $k_L$ de caractéristique $p$. 
Alors pour tout groupe $p$-divisible $G$ sur $S$, muni d'une action de $O_L$ et d'une condition de Pappas-Rapoport $(\omega_{G,\tau}^{[i]})_{\tau,i}$, il existe des applications compatibles au changement de base
\[ \Ha_{\tau}^{[i]}(G) : \det\left(\omega_{G,\tau}^{[i]}/\omega_{G,\tau}^{[i-1]}\right) \fleche \det\left(\omega_{G,\tau}^{[i]}/\omega_{G,\tau}^{[i-1]}\right)^{\otimes(p^f)}\]
De plus, lorsque $S = \Spec(k)$ avec $k$ un corps algébriquement clos de caractéristique $p$, $\Ha_{\tau}^{[i]}$ est inversible si et seulement si les polygones $\Newt(\underline G)$ et $\PR(\mu)$ ont un point de contact en l'abscisse $h-d_{\tau,i}$.
En particulier, si on définit ${^\mu}{\Ha} = \bigotimes_{\tau,i} \Ha_{\tau}^{[i]}$, alors  on a l'équivalence entre les propositions suivantes.
\begin{enumerate}
\item $G$ est $\mu$-ordinaire.
\item ${^\mu}{\Ha}(G)$ est inversible.
\item $G$ est isomorphe à $X^{ord}$.
\item $G[p]$ est isomorphe à $X^{ord}[p]$.
\end{enumerate}
\end{bigtheo}

\begin{bigrema}
En fait la construction des invariants précédents ne dépend que de la $p$-torsion du groupe $p$-divisible $G$. Plus précisément notre construction est valable pour un groupe de Barsotti-Tate tronqué d'échelon 1, muni d'une action de $O_L$, et dont le cristal de Dieudonné est un $\mathcal O_S\otimes_{\ZZ_p}O_L$-module libre (voir remarque \ref{remBTO1}).
\end{bigrema}

 La construction de ces invariants $\Ha_\tau^{[i]}$ est une généralisation de l'invariant de Hasse classique. Dans le cas non ramifié, ils sont aussi construit dans \cite{Her1}, dans \cite{KW} en utilisant les $G$-Zip et \cite{GN} dans le cas des variétés de Shimura en utilisant la cohomologie cristalline.
Ici, c'est l'étude combinatoire de la filtration de Pappas-Rapoport qui permet la construction de tels invariants. Dans \cite{Her1} la construction passe par la cohomologie cristalline et une division du Verschiebung sur le cristal ; ici on adopte une stratégie plus simple : les invariants de Hasse sont alors construit sur le cristal de Berthelot-Breen-Messing de $G$ (cf. \cite{BBM}) mais évalué sur l'épaississement tautologique $(S \fleche S)$, qui est tué par $p$, et l'étape de division du Verschiebung est remplacée par une utilisation astucieuse du Verschiebung et du Frobenius. \\
$ $\\
\indent Décrivons maintenant brièvement l'organisation de l'article. Dans la première partie, nous étudions les $F$-cristaux avec action, satisfaisant une condition de Pappas-Rapoport. Nous définissons les polygones de Hodge, de Newton et de Pappas-Rapoport, et prouvons des inégalités entre ces polygones. Nous prouvons également un analogue du théorème de décomposition Hodge-Newton. Dans la deuxième partie, nous définissons les invariants de Hasse pour un groupe $p$-divisible satisfaisant la condition de Pappas-Rapoport. Nous effectuons cette construction d'abord pour les groupes $p$-divisibles sur un corps parfait, puis dans le cas général. Enfin, nous donnons plusieurs caractérisations des groupes $p$-divisibles $\mu$-ordinaires dans la troisième partie.
\\
$ $

Les auteurs souhaitent remercier Laurent Fargues et Vincent Pilloni pour des échanges intéressants.

\section{Polygones}

\subsection{$F$-cristaux avec action}

Soit $k$ un corps parfait de caractéristique $p>0$, et soit $W(k)$ l'anneau des vecteurs de Witt de $k$. Soit $\sigma$ le Frobenius agissant sur $W(k)$. 

\begin{defi} 
Un $F$-cristal est un couple $(M,F)$, où $M$ est un $W(k)$-module libre de rang fini, et $F : M \to M$ est une injection $\sigma$-linéaire. 
Un morphisme $\varphi : (M,F) \to (M',F')$ entre deux $F$-cristaux est un morphisme de $W(k)$-modules $\varphi : M \to M'$ vérifiant $\varphi \circ F = F' \circ \varphi$.
\end{defi}

\begin{defi}
Un $F$-isocristal est un couple $(N,F)$ où $N$ est un $W(k)[1/p]$-espace vectoriel de dimension finie, et $F : N \to N$ est une bjection $\sigma$-linéaire. Un morphisme $\varphi : (N,F) \to (N',F')$ d'isocristaux est un morphisme de $W(k)[1/p]$-espaces vectoriels $\varphi : N \to N'$ tel que  $\varphi \circ F = F' \circ \varphi$. 
\end{defi}

\begin{rema}
Si l'application $F$ est $\sigma^a$-linéaire pour un certain entier $a \in \mathbb{Z}$, on dit que $(M,F)$ (respectivement $(N,F)$) est un $\sigma^a$-$F$-cristal (resp. isocristal). 
En particulier, à tout $\sigma^a$-$F$-cristal $(M,F)$, on peut associer le $\sigma^a$-$F$-isocristal $(M[1/p],F)$.
\end{rema}

Soit $L$ une extension finie de $\mathbb{Q}_p$, $L^{nr}$ l'extension maximale non-ramifiée contenue dans $L$ et $k_L$ le corps résiduel de $L$. Soient $f$ le degré résiduel, $e$ l'indice de ramification et $\pi$ une uniformisante de $L$. Soient $O_L$ et $O_{L^{nr}}=W(k_L)$ les anneaux des entiers de $L$ et $L^{nr}$, et on suppose que $k$ contient $k_L$.

\begin{defi}
Une action de $O_L$ sur le $F$-cristal $(M,F)$ est la donnée d'un morphisme de $\mathbb{Z}_p$-algèbres
$$O_L \to \text{End} (M,F).$$
\end{defi}

Soit $(M,F)$ un $F$-cristal avec une action de $O_L$. Le module $M$ a en particulier une action de $O_{L^{nr}}$ ; si $\mathcal{T}$ désigne l'ensemble des plongements de $O_{L^{nr}}$ dans $W(k)$, alors on a une décomposition naturelle
$$ M = \bigoplus_{\tau \in \mathcal{T}} M_\tau$$
où $M_\tau$ est le sous-module de $M$ où $O_{L^{nr}}$ agit par $\tau$. Le Frobenius $F$ induit des applications $\sigma$-linéaires
$$F_\tau : M_{\sigma^{-1} \tau} \to M_{\tau}.$$
Les modules $(M_\tau)_{\tau \in \mathcal{T}}$ sont donc des $W(k)$-modules libres ayant le même rang. Définissons pour $\tau \in \mathcal{T}$
$$W_{O_L,\tau} (k) := W(k) \otimes_{O_{L^{nr}},\tau} O_L.$$
L'anneau $W_{O_L,\tau} (k)$ est un anneau de valuation discrète qui a pour uniformisante $\pi$. On normalise la valuation sur cet anneau par $v(p)=1$. Le morphisme $\sigma$ s'étend en un morphisme $W_{O_L,\tau} (k) \to~W_{O_L,\sigma \tau} (k)$ par $\sigma (\pi) = \pi$.

\begin{rema}
De manière explicite, si $E(X) \in O_{L^{nr}} [X]$ est le polynôme minimal de $\pi$ sur $O_{L^{nr}}$, alors pour tout $\tau \in \T$
$$W_{O_L,\tau} (k) = W(k) [X] / (\tau E) (X).$$ 
\end{rema}

\noindent On a des définitions et décompositions analogues pour les isocristaux.

\begin{prop}
Soit $(M,F)$ un $F$-cristal avec une action de $O_L$. Alors $M$ se décompose en $M = \oplus_{\tau \in \mathcal{T}} M_\tau$ ; pour tout $\tau \in \mathcal{T}$, le module $M_\tau$ est un $W_{O_L,\tau} (k)$-module libre de rang fini indépendant de $\tau$. Soit $\tau \in \mathcal{T}$ ; le morphisme $F_\tau$ vérifie
$$F_\tau ( \lambda x) = \sigma ( \lambda) F_\tau (x)$$
pour tout $\lambda \in W_{O_L,\sigma^{-1} \tau} (k)$ et $x \in M_{\sigma^{-1} \tau}$.
\end{prop}

\begin{proof}
Si $(M,F)$ est un $F$-cristal avec une action de $O_L$, alors $M_\tau$ est stable par l'action de $\pi$, donc est un $W_{O_L,\tau} (k)$-module. Puisque $M$ est sans torsion, il est libre sur $W_{O_L,\tau} (k)$. Puisque le rang de $M_\tau$ sur $W(k)$ ne dépend pas de $\tau$, son rang sur $W_{O_L,\tau} (k)$ non plus. \\
On sait déjà que l'on a la relation $F ( \lambda x) = \sigma ( \lambda) F(x)$ pour tout $\lambda \in W (k)$ et $x \in M_\tau$. Soit $[\pi]$ l'action de $\pi$ sur $M$ ; alors $F \circ [\pi] = [\pi] \circ F$. On en déduit que $F ( \lambda x) = \sigma ( \lambda) F(x)$ pour tout $\lambda \in W_{O_L,\sigma^{-1} \tau} (k)$ et $x \in M_{\sigma^{-1} \tau}$.
\end{proof}

Si $(M,F)$ est un $F$-cristal avec une action de $O_L$, alors on définit le rang de ce cristal comme le rang de $M_\tau$ sur $W_{O_L,\tau} (k)$ pour n'importe quel élément $\tau \in \mathcal{T}$. \\
Définissons maintenant les polygones de Hodge et de Newton pour un $F$-cristal $(M,F)$ avec une action de $O_L$. Soit $\tau \in \T$ ; d'après le théorème des diviseurs élémentaires appliqué aux modules $F_\tau M_{\sigma^{-1} \tau} \subset M_\tau$, il existe des éléments $a_{\tau,1}, \dots, a_{\tau,h}$ dans $W_{O_L,\tau} (k)$ tels que
$$M_\tau /F_\tau M_{\sigma^{-1} \tau} \simeq \bigoplus_{i=1}^h W_{O_L,\tau} (k) / a_{\tau,i} W_{O_L,\tau} (k)$$
où $h$ est le rang du cristal. On peut bien sûr supposer les valuations des $a_{\tau , i}$ ordonnés : \\
$v(a_{\tau , 1}) \leq~v(a_{\tau,2}) \leq~\dots \leq~v(a_{\tau,h})$.

\begin{defi}
Le polygone de Hodge de $(M,F)$ relativement au plongement $\tau$ est le polygone à abscisses de rupture entières sur $[0,h]$ défini par $\Hdg_{O_L,\tau} (M,F)(0)=0$ et
$$\Hdg_{O_L,\tau} (M,F)(i) = v(a_{\tau,1}) + \dots + v(a_{\tau,i})$$
pour $1 \leq i \leq h$. Le polygone de Hodge de $(M,F)$ est défini comme la moyenne des polygones $\Hdg_{O_L,\tau} (M,F)$, c'est-à-dire
$$\Hdg_{O_L} (M,F) (i) =\frac{1}{f} \sum_{\tau \in \T} \Hdg_{O_L,\tau} (M,F) (i)$$
pour $0 \leq i \leq h$.
\end{defi}

Les points initial et terminal du polygone $\Hdg_{O_L,\tau} (M,F)$ sont donc $(0,0)$ et $(h, v(\det F_{\tau}))$, où le déterminant de $F_{\tau}$ est vu comme un élément de $W_{O_L,\tau}(k)$. Si $M_\tau / F^f M_\tau \simeq \oplus_{i=1}^h W_{O_L,\tau} (k) / b_{\tau,i} W_{O_L,\tau} (k)$, alors on notera $\Hdg_{O_L} (M_\tau, F^f)$ le polygone convexe ayant pour pentes $v(b_{\tau,1}), \dots, v(b_{\tau,h})$.\\
$ $\\
Soit $(N,F)$ un $F$-isocristal sur $k$, muni d'une action de $O_L$. Soit $\tau \in \T$, et considérons le morphisme $F^f  : N_\tau \to N_\tau$. Quitte à faire une extension $K$ de $W_{O_L,\tau} (k)[1/p]$ en ajoutant des racines de $\pi$, on peut supposer que la matrice de $F^f$ agissant sur $N_\tau$ dans une certaine base est de la forme

 \begin{displaymath}
\left(
\begin{array}{ccc}
\lambda_1 & \dots & \star \\
& \ddots & \vdots \\
& & \lambda_h
\end{array}
\right)
\end{displaymath}

\noindent pour des éléments $\lambda_i \in K$ avec $v(\lambda_1) \leq \dots \leq v(\lambda_h)$. Si $(N,F)$ vérifie $F(N) \subset N$ (par exemple si $(N,F) = (M[1/p],F)$), alors en plus $v(\lambda_i) \geq 0$ pour tout $i$.

\begin{defi}
Le polygone de Newton d'un isocristal $(N,F)$ relativement au plongement $\tau$ est le polygone à abscisses de rupture entières sur $[0,h]$ défini par $\Newt_{O_L,\tau} (N,F)(0)=0$ et
$$\Newt_{O_L,\tau} (N,F)(i) = \frac{v(\lambda_1) + \dots + v(\lambda_i)}{f}$$
pour $1 \leq i \leq h$.
Le polygone de Newton relativement à un plongement $\tau$ d'un cristal $(M,F)$ est celui de l'isocristal $(M[1/p],F)$.
\end{defi}

La théorie de Dieudonné-Manin montre que ce polygone est bien défini, et ne dépend pas de la base dans laquelle la matrice de $F^f$ est écrite. On notera également 
$\Newt_{O_L} (N_\tau, F^f)$ le polygone égal à $f \cdot \Newt_{O_L,\tau} (N,F)$.

\begin{prop}
Soit $(N,F)$ un isocristal. Le polygone $\Newt_{O_L,\tau} (N,F)$ est indépendant de $\tau$, et sera noté $\Newt_{O_L} (N,F)$. C'est le $O_L$-polygone de Newton de $(N,F)$.
\end{prop}

\begin{proof}
Fixons des bases sur les modules $N_\tau$, et soit $X_\tau$ la matrice de $F^f$ agissant sur $N_\tau$ pour tout $\tau$. Soit également $Y_\tau$ la matrice de $F_\tau : N_{\sigma^{-1} \tau} \to N_\tau$. Alors on a l'égalité dans $M_n ( W_{O_L,\tau} (k) [1/p] )$
$$ X_\tau = Y_\tau   (X_{\sigma^{-1} \tau})^\sigma   (Y_\tau^{-1})^{\sigma^f} $$
où $(X_{\sigma^{-1} \tau})^\sigma$ est la matrice $X_{\sigma^{-1} \tau}$ à laquelle on applique $\sigma$ aux coefficients, et de même pour $(Y_\tau^{-1})^{\sigma^f}$. Cela prouve que, quitte à faire un changement de base pour le module $N_\tau$, on peut supposer que la matrice de $F^f$ agissant sur $N_\tau$ est égale à $(X_{\sigma^{-1} \tau})^\sigma $. Les valuations des valeurs propres de $X_\tau$ et $X_{\sigma^{-1} \tau}$ sont donc les mêmes. Puisque $\T$ est un groupe cyclique engendré par $\sigma$, on en déduit le résultat.
\end{proof}

\begin{rema}
On a choisi de détailler la preuve précédente, qui démontre essentiellement que le polygone de Newton d'un cristal sur $W_{O_L,\tau}(k)$ est invariant par isogénie.
\end{rema}

Puisque le morphisme $F^f$ agissant sur $M_\tau$ est égal au composé des morphismes $F_{\sigma^i \tau}$ pour $1 \leq i \leq f$, on voit que les points initial et terminal du polygone $\Newt_{O_L} (M,F)$ sont les mêmes que ceux du polygone $\Hdg_{O_L} (M,F)$.

\begin{ex}
Si $M= W (k) \otimes_{\mathbb{Z}_p} O_L$, et $F= p \cdot (\sigma \otimes 1)$, alors $(M,F)$ est un $F$-cristal avec action de $O_L$ de rang $1$. Les polygones de Hodge et de Newton de $(M,F)$ sont égaux, et ont une pente égale à $1$.
\end{ex}

\begin{rema}
Si $\Newt(M,F)$ et $\Hdg(M,F)$ désignent les polygones de Newton et de de Hodge de $(M,F)$ (sans prendre en compte l'action de $O_L$), alors les pentes de $\Newt (M,F)$ sont exactement celles de $\Newt_{O_L} (M,F)$, chacune étant comptée avec une multiplicité $ef$. En revanche, les pentes de $\Hdg(M,F)$ et $\Hdg_{O_L} (M,F)$ ne sont pas reliées en général.
\end{rema}

\subsection{Condition de Pappas-Rapoport}

Soit $(M,F)$ un $F$-cristal avec action de $O_L$ de rang $h$, et soit pour tout $\tau \in \T$ un entier $r_\tau \geq 1$, ainsi que des entiers $(d_{\tau,1}, \dots, d_{\tau,r_\tau})$ compris entre $0$ et $h$.

\begin{defi}
\phantomsection\label{defPR}
On dit que le $F$-cristal $(M,F)$ satisfait la condition de Pappas-Rapoport (ou condition PR en abrégé) pour la donnée $(d_{\tau,i})_{\tau \in \T, 1 \leq i \leq r_\tau}$ si pour tout $\tau \in \T$, il existe une filtration 
$$F_\tau M_{\sigma^{-1} \tau} = \Fil^{[0]} M_\tau \subset \Fil^{[1]} M_\tau \subset \dots \subset \Fil^{[r_\tau-1]} M_\tau \subset \Fil^{[r_\tau]} M_\tau = M_\tau$$
avec
\begin{itemize}
\item pour tout $0 \leq i \leq r_\tau$, $\Fil^{[i]} M_\tau$ est un sous-$W_{O_L,\tau} (k)$-module de $M_\tau$.
\item on a $\pi \cdot \Fil^{[i]} M_\tau \subset \Fil^{[i-1]} M_\tau$ pour $1 \leq i \leq r_\tau$.
\item $\Fil^{[i]} M_\tau / \Fil^{[i-1]} M_\tau$ est un $k$-espace vectoriel de dimension $d_{\tau,i}$ pour tout $1 \leq i \leq r_\tau$.
\end{itemize}
\end{defi}

Soit $\tau \in \T$. Pour tout $1 \leq i \leq r_\tau$, considérons le polygone convexe défini sur $[0,h]$ ayant pour pentes $0$ avec multiplicité $h-d_{\tau,i}$ et $1/e$ avec multiplicité $d_{\tau,i}$. On définit le polygone de Pappas-Rapoport $\PR_\tau (d_\bullet)$ comme étant la somme de ces polygones pour $1 \leq i \leq r_\tau$. De manière explicite, on a 
$$\PR_\tau (d_\bullet) (j) = \frac{1}{e} \sum_{i=1}^{r_\tau} \max(j-h+d_{\tau,i},0).$$
En étudiant la longueur de $M_\tau / F_\tau M_{\sigma^{-1} \tau}$ comme $W_{O_L,\tau} (k)$-module, on voit que si $(M,F)$ est un $F$-cristal avec action de $O_L$ de rang $h$ satisfaisant la condition PR pour la donnée $(d_\bullet)$, alors $\Hdg_{O_L,\tau} (M,F) (h) =~\PR_\tau (d_\bullet) (h)$. En d'autres termes, les polygones $\Hdg_{O_L,\tau} (M,F)$ et $\PR_\tau (d_\bullet)$ ont mêmes points initial et terminal. \\
On définit le polygone $\PR(d_\bullet)$ comme la moyenne des polygones $\PR_\tau (d_\bullet)$ pour $\tau \in \T$. Les polygones $\Hdg_{O_L} (M,F)$, $\Newt_{O_L} (M,F)$ et $\PR(d_\bullet)$ ont donc mêmes points initial et terminal. \\
$ $\\
Soit $(M,F)$ un $F$-cristal avec action de $O_L$ de rang $h$, et soit $1 \leq s \leq h$ un entier. On définit 
$$\bigwedge^s M = \bigoplus_{\tau \in \T} \bigwedge^s M_\tau$$
le produit extérieur $\bigwedge^s M_\tau$ étant pris sur $W_{O_L,\tau} (k)$. Le morphisme $\bigwedge^s F$ induit des morphismes $\sigma$-linéaires $\bigwedge^s M_{\sigma^{-1} \tau} \to \bigwedge^s M_\tau$. On voit donc que $(\bigwedge^s M, \bigwedge^s F)$ est un $F$-cristal avec action de $O_L$ ; de plus il est de rang $\binom{h}{s}$. 

\begin{rema}
Puisque $M$ est un $W(k) \otimes_{\mathbb{Z}_p} O_L$-module libre de rang $h$, on aurait également pu définir $\bigwedge^s M$ comme le $s$-ième produit extérieur de $M$, celui-ci étant pris sur l'anneau $W(k) \otimes_{\mathbb{Z}_p} O_L$.
\end{rema}

\noindent Supposons maintenant que $(M,F)$ satisfait la condition PR pour une certaine donnée $(d_{\tau,1}, \dots, d_{\tau,r_\tau})_{\tau \in \T}$. On a donc des filtrations pour $\tau \in \T$
$$F_\tau M_{\sigma^{-1} \tau} = \Fil^{[0]} M_\tau \subset \Fil^{[1]} M_\tau \subset \dots \subset \Fil^{[r_\tau-1]} M_\tau \subset \Fil^{[r_\tau]} M_\tau = M_\tau.$$

\begin{defi}
On définit une filtration sur $\bigwedge^s M_\tau$ par 
$$ \Fil^{[is + j]} \bigwedge^s M_\tau := \im  \left( \bigwedge^{s-j} \Fil^{[i]} M_\tau \otimes \bigwedge^{j} \Fil^{[i+1]} M_\tau \to \bigwedge^s M_\tau \right)$$
pour $0 \leq i \leq r_\tau-1$ et $0 \leq j < s$. On définit également $\Fil^{[r_\tau s]} \bigwedge^s M_\tau := \bigwedge^{s} \Fil^{[r_\tau]} M_\tau$.
\end{defi}

On a donc une filtration
$$\bigwedge^s F_\tau M_{\sigma^{-1} \tau} = \Fil^{[0]} \bigwedge^s M_\tau \subset \Fil^{[1]} \bigwedge^s M_\tau \subset \dots \subset \Fil^{[r_\tau s-1]} \bigwedge^s M_\tau \subset \Fil^{[r_\tau s]} \bigwedge^s M_\tau = \bigwedge^s M_\tau.$$

On vérifie facilement que $\pi \cdot \Fil^{[i]} \bigwedge^s M_\tau \subset \Fil^{[i-1]} \bigwedge^s M_\tau$, pour $1 \leq i \leq r_\tau s$. De plus, il est possible de calculer la dimension des $k$-espaces vectoriels $\left(\Fil^{[i]} \bigwedge^s M_\tau\right) / \left(\Fil^{[i-1]} \bigwedge^s M_\tau\right)$, pour $1 \leq i \leq r_\tau s$.

\begin{prop}
Soit $i$ un entier entre $1$ et $r_\tau$, et $k$ un entier compris entre $0$ et $s-1$. Alors la dimension sur $k$ de $\left(\Fil^{[is - k]} \bigwedge^s M_\tau\right) /\left( \Fil^{[is-k-1]} \bigwedge^s M_\tau\right)$ est égale à
$$\mathlarger{\mathlarger{\sum}}_{j=0}^k \binom{d_{\tau,i}}{s-j} \binom{h-d_{\tau,i}}{j}.$$
\end{prop}

\begin{proof}
Posons $N = \Fil^{[i]} M_\tau$, et $N_0 = \Fil^{[i-1]} M_\tau$. Alors $N$ et $N_0$ sont des $W_{O_L,\tau} (k)$-modules libres de rang $h$, $\pi \cdot N \subset N_0 \subset N$, et le $k$-espace vectoriel $N/N_0$ est de dimension $d_{\tau,i}$. On peut donc choisir une base $(e_1, \dots, e_h)$ de $N$ sur $W_{O_L,\tau} (k)$ telle que $N_0$ soit égal au module engendré par $(\pi e_1, \dots, \pi e_{d_{\tau,i}}, e_{d_{\tau,i}+1}, \dots, e_h)$. Le module $\bigwedge^s N$ est engendré par les éléments $e_{i_1} \wedge \dots\wedge e_{i_s}$, avec $i_1 < \dots < i_s$. De plus, pour tout entier $0 \leq j \leq s$, le module $N_j := $ Im $(\bigwedge^j N_0 \otimes \bigwedge^{s-j} N \to \bigwedge^s N)$ est engendré par
\begin{itemize}
\item $\pi^j e_{i_1} \wedge \dots \wedge e_{i_s}$, avec $i_1 < \dots < i_s \leq d_{\tau,i}$.
\item $\pi^{j-1} e_{i_1} \wedge \dots \wedge e_{i_s}$, avec $i_1 < \dots < i_{s-1} \leq d_{\tau,i} < i_s$.
\item \dots
\item $\pi e_{i_1} \wedge \dots \wedge e_{i_s}$, avec $i_1 < \dots < i_{s-j+1} \leq d_{\tau,i} < i_{s-j} < \dots < i_s$.
\item $e_{i_1} \wedge \dots \wedge e_{i_s}$, avec $i_1 < \dots < i_s$ et $d_{\tau,i} < i_{s-j+1}$.
\end{itemize}
Le quotient $\left(\Fil^{[is - k]} \bigwedge^s M_\tau\right) / \left(\Fil^{[is-k-1]} \bigwedge^s M_\tau\right)$ étant égal à $N_k / N_{k+1}$, on en déduit le résultat.
\end{proof}

Définissons des entiers $(d_{\tau,l}^{(s)})_{\tau \in \T,1 \leq l \leq r_\tau s}$ par la formule
$$d_{\tau,is - k}^{(s)} = \mathlarger{\mathlarger{\sum}}_{j=0}^k \binom{d_{\tau,i}}{s-j} \binom{h-d_{\tau,i}}{j}$$
pour $\tau \in \T$, $1 \leq i \leq r_\tau$ et $0 \leq k \leq s-1$.

\begin{coro}
Le $F$-cristal $(\bigwedge^s M, \bigwedge^s F)$ satisfait la condition PR pour la donnée $(d_{\tau,l}^{(s)})_{\tau \in \T,1 \leq l \leq r_\tau s}$.
\end{coro}

Il est possible de relier les différents polygones de $(M,F)$ à ceux de $(\bigwedge^s M, \bigwedge^s F)$.

\begin{prop} \phantomsection\label{ext}
Soit $(M,F)$ un $F$-cristal avec une action de $O_L$ de rang $h$. On suppose que $(M,F)$ satisfait la condition PR pour la donnée $(d_{\tau,1}, \dots, d_{\tau,r_\tau})_{\tau \in \T}$. Alors on a pour tout entier $1 \leq s \leq h$ et $\tau \in \T$
\begin{itemize}
\item $\Newt_{O_L} (\bigwedge^s M, \bigwedge^s F) (1) = \Newt_{O_L} (M,F) (s)$.
\item $\Hdg_{O_L,\tau} (\bigwedge^s M, \bigwedge^s F) (1) = \Hdg_{O_L,\tau} (M,F) (s)$.
\item $\PR_\tau (d_\bullet^{(s)}) (1) = \PR_\tau (d_\bullet) (s)$.
\end{itemize}
En particulier, on en déduit, $\Hdg_{O_L} (\bigwedge^s M, \bigwedge^s F) (1) = \Hdg_{O_L} (M,F) (s)$ pour tout $1 \leq s \leq h$.
\end{prop}

\begin{proof}
Montrons la première égalité. Soit $\tau$ un élément de $\T$, et soit $\lambda_1, \dots, \lambda_h$ les valeurs propres de $F^f$ agissant sur $M_\tau$ dans une certaine base. Si on ordonne les valeurs propres $\lambda_i$ de telle sorte que $v(\lambda_1) \leq \dots \leq v(\lambda_h)$, alors $\Newt_{O_L} (M,F) (s) = (v(\lambda_1) + \dots + v(\lambda_s))/f$. Les valeurs propres de $\bigwedge^s F^f$ agissant sur $M_\tau$ sont $\lambda_{i_1} \cdot \dots \cdot \lambda_{i_s}$ avec $i_1 < \dots <i_s$. D'où
$$\Newt_{O_L} (\bigwedge^s M, \bigwedge^s F) (1) = \frac{v(\lambda_1 \dots \lambda_s)}{f} = \Newt_{O_L} (M,F) (s).$$
Cela prouve le premier point. Pour le deuxième, étudions le morphisme $F_\tau : M_{\sigma^{-1} \tau} \to M_\tau$. Quitte à faire des changements de base pour ces deux modules, on peut supposer que la matrice de $F_\tau$ est diagonale, avec coefficients $a_1, \dots, a_h$. On ordonne les $a_i$ de telle sorte que $v(a_1) \leq \dots \leq v(a_h)$. Alors $\Hdg_{O_L,\tau} (M,F) (s) = v(a_1) + \dots + v(a_s)$. La matrice de $\bigwedge^s F_\tau$ est également diagonale (pour certaines bases de $\bigwedge^s M_{\sigma^{-1} \tau}$ et $\bigwedge^s M_\tau$), avec coefficients $a_{i_1} \dots a_{i_s}$, pour $i_1 < \dots <i_s$. On en déduit que 
$$\Hdg_{O_L,\tau} (\bigwedge^s M, \bigwedge^s F) (1) = v( a_1 \dots a_s) = \Hdg_{O_L,\tau} (M,F) (s).$$
Montrons maintenant la troisième égalité. La quantité $\PR_\tau(d_\bullet) (1)$ est égale au nombre de $d_{\tau,i}$ égaux à $h$ divisés par $e$. Pour calculer le terme $\PR_\tau(d_\bullet^{(s)}) (1)$, il suffit donc de déterminer quels éléments $d_{\tau,k}^{(s)}$ sont égaux à $\binom{h}{s}$. \\
Fixons un entier $1 \leq i \leq r_\tau$. Si $k$ est compris entre $0$ et $s-1$, nous affirmons que $d_{\tau,is-k}^{(s)} = \binom{h}{s}$ si et seulement si $k \geq h-d_{\tau,i}$. En effet, si $k \geq h - d_{\tau,i}$, on a 
$$d_{\tau,is - k}^{(s)} = \mathlarger{\mathlarger{\sum}}_{j=0}^{h-d_{\tau,i}} \binom{d_{\tau,i}}{s-j} \binom{h-d_{\tau,i}}{j} = \binom{h}{s}$$
la dernière égalité étant obtenue en développant $(X+Y)^h = (X+Y)^{d_{\tau,i}} (X+Y)^{h-d_{\tau,i}}$. Si $k < h - d_{\tau,i}$, nous voulons prouver que
$$d_{\tau,is - k}^{(s)} = \mathlarger{\mathlarger{\sum}}_{j=0}^k \binom{d_{\tau,i}}{s-j} \binom{h-d_{\tau,i}}{j} < \mathlarger{\mathlarger{\sum}}_{j=0}^{h-d_{\tau,i}} \binom{d_{\tau,i}}{s-j} \binom{h-d_{\tau,i}}{j}.$$
Il suffit pour cela de prouver qu'il existe un entier $j$ avec $k+1 \leq j \leq h-d_{\tau,i}$ et $0 \leq s-j \leq d_{\tau,i}$. Puisque l'on a $h-d_{\tau,i} \geq s-d_{\tau,i}$ et $k+1 \leq s$, cela est en effet possible. \\
Le nombre d'entiers $0 \leq k \leq s-1$ tels que $d_{\tau,is - k}^{(s)} = \binom{h}{s}$ est donc égal à $\max(s-h+d_{\tau,i},0)$. On en déduit que
$$\PR_\tau (d_\bullet^{(s)}) (1) = \frac{1}{e} \sum_{i=1}^{r_\tau} \max(s-h+d_{\tau,i},0) = \PR_\tau (d_\bullet) (s).$$
\end{proof}

\subsection{Propriétés}

Si $(M,F)$ est un $F$-cristal avec une action de $O_L$ de rang $h$ satisfaisant la condition PR pour une donnée $(d_\bullet)$, alors nous avons construit trois polygones : le polygone de Hodge, le polygone de Newton et le polygone de Pappas-Rapoport. Nous avons de plus des relations entre ces trois polygones. Si $P_1$ et $P_2$ sont deux polygones à abscisses de ruptures entières définis sur $[0,h]$, on dit que $P_1 \geq P_2$ si $P_1(i) \geq P_2(i)$ pour $0 \leq i \leq h$.

\begin{theo}
\phantomsection\label{thr118}
Soit $(M,F)$ un $F$-cristal avec une action de $O_L$ de rang $h$ satisfaisant la condition PR pour une donnée $(d_\bullet)$. Alors on a les relations
$$\Newt_{O_L} (M,F) \geq \Hdg_{O_L} (M,F) \geq \PR(d_\bullet).$$
\end{theo}

\begin{proof}
Fixons un élément $\tau \in \T$. On rappelle que $\Newt_{O_L} (M,F) = \Newt_{O_L} (M_\tau,F^f)/f$ est le polygone de Newton obtenu à partir du morphisme $F^f$ agissant sur $M_{\tau}$. Nous allons montrer que l'on a 
$$\Newt_{O_L} (M_\tau,F^f) \geq \Hdg_{O_L} (M_\tau,F^f).$$
En effet, on peut supposer que la matrice de $F^f$ agissant sur $M_\tau$ est triangulaire supérieure, avec des coefficients diagonaux $\lambda_1, \dots, \lambda_h$ vérifiant $v(\lambda_1) \leq \dots \leq v(\lambda_h)$. Si on écrit $(z_{i,j})$ les coefficients de $F^f$, alors
$$\Hdg_{O_L} (M_\tau,F^f) (1) = \inf_{i,j} v(z_{i,j}) \leq v(\lambda_1) = \Newt_{O_L} (M_\tau,F^f) (1).$$
En appliquant le résultat précédent à $(\bigwedge^s M, \bigwedge^s F)$, on en déduit que le polygone $\Newt_{O_L} (M_\tau,F^f)$ est au-dessus du polygone $\Hdg_{O_L} (M_\tau,F^f)$. \\
Nous allons maintenant montrer que le polygone $\Hdg_{O_L} (M_\tau,F^f)$ est au-dessus du polygone $f \cdot \Hdg_{O_L} (M,F)$. Fixons des bases pour les modules $M_{\tau'}$ (sur $W_{O_L,\tau'} (k)$) pour tout $\tau' \in \T$, et soit $Y_{\tau'} = (y_{\tau',i,j})_{1 \leq i,j \leq h}$ la matrice de $F_{\tau'} : M_{\sigma^{-1} \tau'} \to M_{\tau'}$. Alors la matrice de $F^f$ agissant sur $M_\tau$ est 
$$ X_\tau = \prod_{i=1} ^ f Y_{\sigma^i \tau}^{\sigma^{i-1}}  $$
où $Y_{\sigma^i \tau}^{\sigma^{i-1}}$ est la matrice avec coefficients $(\sigma^{i-1} (y_{\sigma^i \tau,i,j}))_{i,j}$. Soit $(x_{\tau,i,j})_{i,j}$ les coefficients de la matrice de $X_\tau$. Alors on a $\Hdg_{O_L} (M_\tau,F^f) (1) = \inf_{i,j} v(x_{\tau,i,j})$ et $\Hdg_{O_L,\tau'} (M,F) (1) = \inf_{i,j} v(y_{\tau',i,j})$ pour tout $\tau' \in \T$. Cela prouve que
$$\Hdg_{O_L} (M_\tau,F^f) (1) \geq f \cdot \Hdg_{O_L} (M,F) (1).$$
Soit maintenant $s$ un entier entre $1$ et $h$. En appliquant le résultat précédent à $(\bigwedge^s M, \bigwedge^s F)$, on obtient que $\Hdg_{O_L} (\bigwedge^s M_\tau,(\bigwedge^s F)^f) (1) \geq f \cdot \Hdg_{O_L} (\bigwedge^s M,\bigwedge^s F) (1)$. Mais d'apr\`es la proposition $\ref{ext}$, on a $\Hdg_{O_L} (\bigwedge^s M,\bigwedge^s F) (1) = \Hdg_{O_L} (M,F) (s)$ ; en appliquant cette m\^eme proposition au $\sigma^f$-F-cristal $(M_\tau,F^f)$, on obtient 
$$\Hdg_{O_L} (\bigwedge^s M_\tau,(\bigwedge^s F)^f) (1) = \Hdg_{O_L} (M_\tau, F^f) (s)$$
Cela prouve donc l'inégalité entre polygones
$$\Hdg_{O_L} (M_\tau,F^f) \geq f \cdot  \Hdg_{O_L} (M,F)$$
En conclusion, nous avons les inégalités
$$\Newt_{O_L} (M,F) = \frac{1}{f} \Newt_{O_L} (M_\tau,F^f) \geq \frac{1}{f} \Hdg_{O_L} (M_\tau,F^f) \geq \Hdg_{O_L} (M,F).$$
Montrons enfin que le polygone $\Hdg_{O_L} (M,F)$ est au-dessus du polygone $\PR(d_\bullet)$. Pour cela, nous allons montrer que le polygone $\Hdg_{O_L,\tau} (M,F)$ est au-dessus du polygone $\PR_\tau (d_\bullet)$ pour tout $\tau \in \T$. Soit donc $\tau$ un élément de $\T$. Nous devons prouver que pour tout entier $1 \leq s \leq h$, on a 
$$\Hdg_{O_L,\tau} (M,F) (s) \geq \PR_\tau (d_\bullet) (s).$$
Prouvons tout d'abord le résultat pour $s=1$. Soit $j = e \PR_\tau (d_\bullet) (1)$. Alors $j$ est égal au nombre de $d_{\tau,k}$ égaux à $h$. On en déduit que $F_\tau M_{\sigma^{-1} \tau} \subset \pi^j M_\tau$. Or on peut trouver une base $(e_1, \dots, e_h)$ de $M_\tau$ telle que $(\pi^{a_1} e_1, \dots, \pi^{a_h} e_h)$ soit une base de $F_\tau M_{\sigma^{-1} \tau}$, et $e \Hdg_{O_L,\tau } (M,F) (1) = \min_k a_k$. L'inclusion $F_\tau M_{\sigma^{-1} \tau} \subset \pi^j M_\tau$ implique que $a_k \geq j$ pour tout $1 \leq k \leq h$, et donc $e \Hdg_{O_L,\tau} (M,F) (1) \geq e \PR_\tau (d_\bullet) (1)$. \\
Soit maintenant $s$ un entier compris entre $1$ et $h$. En appliquant le résultat précédent au cristal $(\bigwedge^s M, \bigwedge^s F)$, on obtient
$$\Hdg_{O_L,\tau} (\bigwedge^s M, \bigwedge^s F) (1) \geq \PR_\tau(d_\bullet^{(s)}) (1).$$  
La proposition $\ref{ext}$ permet de conclure.
\end{proof}

Le polygone de Newton est donc toujours au-dessus du polygone de Hodge. De plus, le fait que ces deux polygones ont un point commun implique une condition très forte sur le cristal. En effet, le théorème de décomposition Hodge-Newton reste valable dans ce cadre.

\begin{theo}
\phantomsection\label{thr119}
Soit $(M,F)$ un $F$-cristal avec action de $O_L$ de rang $h$. Soient $a_{\tau,1} \leq \dots \leq a_{\tau,h}$ les pentes du polygones $\Hdg_{O_L,\tau} (M,F)$ pour $\tau \in \mathcal{T}$ et $\lambda_1 \leq \dots \leq \lambda_h$ celles de $\Newt_{O_L} (M,F)$. On suppose qu'il existe un point $(i,j) \dans \NN \times \frac{1}{ef}\NN$ qui est un point de rupture du polygone de Newton $\Newt_{O_L} (M,F)$ qui se trouve également sur le polygone $\Hdg_{O_L} (M,F)$. Alors il existe une unique décomposition de $(M,F)$ en somme directe $(M,F) = (M_1,F_1) \oplus (M_2,F_2)$, où $(M_i,F_i)$ sont des $F$-cristaux avec action de $O_L$, et tels que
\begin{itemize}
\item $M_1$ est de rang $i$, le polygone $\Newt_{O_L}(M_1,F_1)$ a pour pentes $\lambda_1, \dots, \lambda_i$ et le polygone $\Hdg_{O_L,\tau} (M_1,F_1)$ a pour pentes $a_{\tau,1}, \dots, a_{\tau,i}$ pour tout $\tau \in \mathcal{T}$.
\item $M_2$ est de rang $h-i$, le polygone $\Newt_{O_L}(M_2,F_2)$ a pour pentes $\lambda_{i+1}, \dots, \lambda_h$ et le polygone $\Hdg_{O_L,\tau} (M_2,F_2)$ a pour pentes $a_{\tau,i+1}, \dots, a_{\tau,h}$ pour tout $\tau \in \mathcal{T}$.
\end{itemize}
\end{theo}

\begin{proof}
Ce théorème est analogue au théorème obtenu par Katz (\cite{Katz} Théor\`eme 1.6.1) dans le cas des cristaux sans action, et la démonstration est similaire. Soit $\tau \in \mathcal{T}$, et considérons le cristal $(M_\tau, F^f)$. On a les inégalités
$$\Newt_{O_L} (M_\tau,F^f) = f \cdot \Newt_{O_L} (M,F) \geq \Hdg_{O_L}(M_\tau,F^f) \geq f \cdot \Hdg_{O_L} (M,F) $$
Le point $(i,fj)$ est donc un point de rupture du polygone $\Newt_{O_L}(M_\tau,F^f)$ qui est également sur le polygone $\Hdg_{O_L}(M_\tau,F^f)$. Le résultat de Katz \cite{Katz} Théor\`eme 1.6.1 montre que l'on peut décomposer de manière unique ce cristal en $M_\tau = M_{\tau,1} \oplus M_{\tau,2}$, avec $M_{\tau,i}$ stable par $F^f$. De plus les pentes du polygone $\Newt_{O_L}(M_{\tau,1},F^f)$ (resp. du polygone $\Hdg_{O_L}(M_{\tau,1},F^f)$) sont égales aux $i$ premières pentes du polygone $\Newt_{O_L}(M_\tau,F^f)$ (resp. du polygone $\Hdg_{O_L}(M_\tau,F^f)$). On rappelle les principales étapes de la preuve de ce résultat. \\
On traite tout d'abord le cas où $i=1$. Dans ce cas, le morphisme $F^f$ est divisible par $\pi^{efj}$, et quitte à diviser par cet élément, on se ramène au cas où $j=0$. On montre alors qu'il existe une unique droite $L \subset M_\tau$ telle que $F$ agisse par un automorphisme sur $L$. De manière explicite, on a 
$$L / p^n L = \bigcap_{m \geq 0} F^{mf} (M_\tau / p^n M_\tau) $$
On peut ensuite montrer que la droite $L$ admet un unique supplémentaire dans $M_\tau$ stable par $F^f$, ce qui prouve le résultat annoncé pour $i=1$. \\
Pour le cas général, on considère le cristal $(\bigwedge^i M_\tau, \bigwedge^i F^f)$. D'après le cas précédent, il existe une unique décomposition $\bigwedge^i M_\tau = N_1 \oplus N_2$, avec chaque $N_i$ stable par $F^f$, avec la propriété que $N_1$ est de rang $1$, et la pente des polygones de Newton et de Hodge de $(N_1,F^f)$ est égale à $j$. On veut prouver qu'il existe $M_{\tau,1} \subset M$ tel que $N_1 = \bigwedge^i M_{\tau,1}$. Il suffit pour cela de prouver que $N_1$ satisfait les équations de Plücker. Mais après inversion de $p$, la matrice de $F^f$ sur $M_\tau$ peut être diagonalisée donc $N_1[1/p]$ est uniquement déterminé et satisfait les équations de Plücker. On en déduit l'existence de $M_{\tau,1} \subset M_\tau$ de rang $i$ et stable par $F^f$, avec les pentes du polygone $\Newt_{O_L}(M_{\tau,1},F^f)$ (resp. du polygone $\Hdg_{O_L}(M_{\tau,1},F^f)$) sont égales aux $i$ premières pentes du polygone $\Newt_{O_L}(M_\tau,F^f)$ (resp. du polygone $\Hdg_{O_L}(M_\tau,F^f)$). De plus, ce sous-module est uniquement déterminé, et on peut prouver qu'il existe un unique supplémentaire $M_{\tau,2}$ stable par $F^f$. \\
On définit ensuite pour $i \in \{1,2\}$
$$M_i = \bigoplus_{\tau \in \mathcal{T}} M_{\tau,i}$$
Comme le Frobenius $F : M_{\sigma^{-1} \tau} \to M_{\tau}$ est une isogénie, il stabilise les polygones de Newton, et donc il envoie $M_{\sigma^{-1} \tau,i}$ dans $M_{\tau,i}$ pour $\tau \in \mathcal{T}$ et $i \in \{1,2\}$. En d'autres termes, les modules $M_1$ et $M_2$ sont stables par le Frobenius. La partie sur les polygones de Newton de ces modules est claire. De plus, on a pour tout $\tau \in \mathcal{T}$ l'inégalité $\Hdg_{O_L,\tau} (M_1,F) (i) \geq \Hdg_{O_L,\tau} (M,F) (i)$. Si $\tau_0$ est un élément fixé de $\mathcal{T}$, on a alors
$$ f \cdot j = \Hdg_{O_L} (M_{\tau_0},F^f) (i) = \Hdg_{O_L} (M_{\tau_0,1},F^f) (i) \geq \sum_{\tau \in \mathcal{T}} \Hdg_{O_L,\tau} (M_1,F) (i) \geq \sum_{\tau \in \mathcal{T}} \Hdg_{O_L,\tau} (M,F) (i) = f \cdot j$$
On a donc les égalités $\Hdg_{O_L,\tau} (M_1,F) (i) = \Hdg_{O_L,\tau} (M,F) (i)$, ce qui prouve que les pentes du polygone $\Hdg_{O_L,\tau} (M_1,F)$ sont égales aux $i$ premières pentes du polygone $\Hdg_{O_L,\tau} (M,F)$ pour tout $\tau \in \mathcal{T}$.
\end{proof}

\section{Invariants de Hasse $\mu$-ordinaires}
\phantomsection\label{sect3}
\subsection{Construction sur un corps parfait} \phantomsection\label{corps}

On suppose toujours donné une extension finie $L$ de $\mathbb{Q}_p$, et soit $k$ un corps parfait de caractéristique $p$ contenant le corps résiduel de $L$. Soit $G \to \Spec(k)$ un groupe $p$-divisible, et on suppose $G$ muni d'une action de $O_L$, c'est à dire qu'il existe un morphisme $\ZZ_p$-linéaire,
\[ \iota_G : O_L \fleche \End(G)\]

Soit $(M,F,V)$ le module de Dieudonné (contravariant) de $G$ (\cite{Fon} partie III), alors $(M,F)$ est un $F$-cristal avec action de $O_L$ au sens de la section précédente. 
En particulier on peut associer à $(M,F)$, et donc à $G$, des polygones de Hodge et Newton, $\Hdg_{O_L}(G),\Newt_{O_L}(G)$ qui ont même points initial $(0,0)$ 
et mêmes points terminaux, et tels que le polygone de Newton est au-dessus du polygone de Hodge. \\
L'algèbre de Lie duale $\omega_G$ de $G$ est un $k$-espace vectoriel muni d'une action de $O_L$, et se décompose en
$$\omega_G = \bigoplus_{\tau \in \T} \omega_{G,\tau}$$
On rappelle que $\T$ est l'ensemble des plongements de l'extension maximale non ramifiée $L^{nr}$ de $L$ dans $\overline{\QQ_p}$, et l'anneau des entiers de $L^{nr}$ agit par $\tau$ sur $\omega_{G,\tau}$. On se donne une collection d'entiers $\mu = (d_{\tau,i})_{\tau \in \T,1 \leq i \leq e}$. La condition de Pappas-Rapoport pour le groupe $p$-divisible $G$ est définie comme suit.

\begin{defi}
\phantomsection\label{defPRG}
On dit que $G$ satisfait la condition de Pappas-Rapoport (PR en abrégé) pour la donnée $\mu$ s'il existe une filtration pour tout $\tau \in \T$
\[ 0=\omega_{G,\tau}^{[0]} \subset \omega_{G,\tau}^{[1]} \subset \omega_{G,\tau}^{[2]} \subset \dots \subset \omega_{G,\tau}^{[e]} = \omega_{G,\tau}\]
telle que
\begin{itemize}
\item $\omega_{G,\tau}^{[i]}$ est un sous-espace vectoriel de $\omega_{G,\tau}$ pour tout $1 \leq i \leq e$ et $\tau \in \T$.
\item $\pi \cdot \omega_{G,\tau}^{[i]} \subset \omega_{G,\tau}^{[i-1]}$ pour tout $1 \leq i \leq e$ et $\tau \in \T$.
\item si on note $\Gr^{[i]}\omega_{G,\tau} := \omega_{G,\tau}^{[i]}/\omega_{G,\tau}^{[i-1]}$, alors $\dim_k \Gr^{[i]}\omega_{G,\tau} = d_{\tau,i}$ pour tout $1 \leq i \leq e$ et $\tau \in \T$.
\end{itemize}
\end{defi}

Comme $M/FM \simeq \omega_G$, la condition précédente est équivalente au fait que le $F$-cristal $(M,F)$ (le module de Dieudonné de $G$) vérifie la condition de Pappas-Rapoport pour la donnée $(d_{\tau,i})_{\tau \in \mathcal T,1 \leq i \leq e}$. \\
On souhaite définir des sections de certaines puissances du faisceau $\omega_G$. Pour se faire, il est possible de travailler avec le module de Dieudonné $M$ en utilisant l'isomorphisme précédent. Cependant, dans le cas général, nous serons amené à utiliser le cristal de Dieudonné de $G$ et par souci de clarté, nous travaillerons également avec le cristal de Dieudonné dans cette partie. Soit $D$ le cristal de Dieudonné (\cite{BBM} partie $3.3$) évalué sur l'anneau $W(k)$ (muni de ses puissances divisées canoniques). C'est un $W(k)$-module libre muni d'une action de $O_L$, et il est relié simplement au module de Dieudonné usuel par la formule $D \simeq M^{(\sigma)}$, où l'exposant signifie un twist par le Frobenius (\cite{BBM} théor\`eme $4.2.14$). Ce module se décompose en $D = \oplus_{\tau \in \T} D_\tau$ ; de plus le Frobenius induit des morphismes $\sigma$-linéaires $F_\tau : D_{\sigma^{-1} \tau} \to D_\tau$, et le Verschiebung des morphismes $\sigma^{-1}$-linéaires $V_\tau : D_\tau \to D_{\sigma^{-1} \tau}$. Ces morphismes vérifient $F_\tau \circ V_\tau = p \Id_{D_\tau}$ et $V_\tau \circ F_\tau = p \Id_{D_{\sigma^{-1} \tau}}$. De plus, on a $V D / p D \simeq \omega_G$, et $V_{\sigma \tau} D_{\sigma \tau} / p D_\tau \simeq \omega_{G,\tau}$ pour tout $\tau \in \T$.

On note pour tout $\tau \in \T$ et $1 \leq i \leq e$,
\[ \Fil^{[i]} D_\tau = \phi_\tau^{-1}(\omega_{G,\tau}^{[i]})\]
où $\phi_\tau : D_\tau \fleche D_\tau/pD_\tau$ est la réduction modulo $p$, et où on voit $\omega_{G,\tau}$ comme un sous objet de $D_\tau/pD_\tau$, égal à $V_{\sigma \tau} D_{\sigma \tau}/pD_\tau$. En particulier, $\Fil^{[e]} D_\tau = V_{\sigma \tau} D_{\sigma\tau}$. 
La condition sur $\pi$ sur l'algèbre de Lie assure que $\pi^j(\Fil^{[j]}D_\tau) \subset pD_\tau$ pour tout $j,\tau$.
Comme $D_\tau$ est un $W_{O_L,\tau}(k)$-module libre, on en déduit le lemme suivant,

\begin{lemm} 
Soit $\tau \in \T$ et $1 \leq i \leq e$. Alors $\Fil^{[i]}D_\tau \subset \pi^{e-i}D_\tau$.
\end{lemm}

On va construire des morphismes sur des puissances extérieures des modules précédent. Si $\tau \in \T$ et $d$ est un entier plus grand que $1$, on considère le module $\bigwedge^d D_\tau$, le produit extérieur étant pris sur $W_{O_L,\tau}(k)$ (ce sera le cas de tous les produits extérieurs dans cette section). Puisque $D_\tau$ est un $W_{O_L,\tau}(k)$-module libre, on a en particulier une inclusion pour $1 \leq i \leq e$,
\[ \bigwedge^d \Fil^{[i]} D_\tau \subset \bigwedge^d D_\tau \]

Avant de poursuivre, commençons par le lemme suivant.

\begin{lemm}
Soit $\tau \in \T$, $M$ un $W_{O_L,\tau}(k)$-module libre de rang fini $r$ et $N \subset M$ un sous-module libre. On suppose que $\pi M \subset N$, et notons $\dim_k M/N = f$. Alors pour $d \dans \NN$, on a $\pi^{\min(d,f)} \bigwedge^d M \subset~\bigwedge^d N$.
\end{lemm}

\begin{proof}
Il existe une base $(e_1,\dots,e_r)$ de $M$ tel que $(\pi e_1,\dots,\pi e_f,e_{f+1},\dots,e_r)$ soit une base de $N$. Si $d \leq f$ le résultat est évident ; supposons donc $d > f$. Alors $\bigwedge^d M$ est engendré par les images des tenseurs élémentaires $e_{j_1}\wedge \dots \wedge e_{j_d}$, avec 
$1 \leq j_1 < j_2 < \dots < j_d \leq r$.
En particulier, $|\{k : j_k \leq f\}| \leq f$ et donc $\pi^f(e_{j_1}\wedge \dots \wedge e_{j_d}) \dans \bigwedge^d N$. On en déduit le résultat.
\end{proof}



On en déduit le corollaire suivant,

\begin{coro}
\phantomsection\label{corpi}
Soit $\tau \in \T$ et $1 \leq i \leq e$. La multiplication par  $\pi^{\min(d,d_{\tau,i})}$ envoie $\bigwedge^d\Fil^{[i]}D_\tau$ dans $\bigwedge^d\Fil^{[i-1]}D_\tau$.
\end{coro}


\begin{prop} 
\phantomsection\label{pro35}
Soit $\tau \in \T$ et $d \geq 1$ un entier. Alors il existe une application $\sigma^{-1}$-linéaire,
\[ \zeta_\tau^d : \bigwedge^d D_{\sigma\tau} \fleche \bigwedge^dD_\tau\]
telle que si 
\[ k_{\tau,d} = \sum_{i=1}^e \max(d-d_{\tau,i},0)\]
alors $\pi^{k_{\tau,d}}\zeta_\tau^d = \bigwedge^d V_{\sigma \tau}$.
\end{prop}

\begin{proof} 
En effet, l'application 
\[ \bigwedge^d V_{\sigma \tau} : \bigwedge^d D_{\sigma\tau} \fleche \bigwedge^d D_\tau\]
a son image dans $\bigwedge^d\Fil^{[e]}D_\tau$.
D'après le corollaire précédent, pour tout $i \dans \{1,\dots,e\}$, l'application $\pi^{\min(d,d_{\tau,i})}$ envoie
$\bigwedge^d\Fil^{[i]}D_\tau$ dans $\bigwedge^d\Fil^{[i-1]}D_\tau$.
En particulier, on peut faire la composée suivante, notée $\zeta_\tau^d$,

\begin{center}
\begin{tikzpicture}[description/.style={fill=white,inner sep=2pt}] 
\matrix (m) [matrix of math nodes, row sep=3em, column sep=2.5em, text height=1.5ex, text depth=0.25ex] at (0,0)
{ 
\bigwedge^{d} D_{\sigma\tau} & & & \bigwedge^{d} \Fil^{[e]} D_\tau & & \bigwedge^{d} D_\tau \\
& & & \bigwedge^{d} \Fil^{[e-1]} D_\tau& & \\
& & & \bigwedge^{d} \Fil^{[1]} D_\tau& & \\
 & & & \bigwedge^{d} \Fil^{[0]}D_\tau & & \\
 };

\path[->,font=\scriptsize] 
(m-1-1) edge node[auto] {$\bigwedge^{d} V_{\sigma \tau}$} (m-1-4)
(m-1-4) edge node[auto] {$\pi^{\min(d,d_{\tau,e})}$} (m-2-4)
(m-3-4) edge node[auto,left] {$\pi^{\min(d,d_{\tau,1})}$} (m-4-4)
(m-4-4) edge node[auto,right] {$\frac{1}{\pi^{ed}}$} (m-1-6)
;
\path[dashed,->,font=\scriptsize] 
(m-2-4) edge node[auto,left] {$\pi^{\min(d,d_{\tau,2})} \circ \dots \circ \pi^{\min(d,d_{\tau,e-1})}$} (m-3-4)
;
\end{tikzpicture}
\end{center}

En effet, $\bigwedge^{d} \Fil^{[0]}D_\tau = \pi^{ed}\bigwedge^d D_\tau$ est un sous-module d'un $W_{O_L,\tau}(k)$-module libre, 
on peut donc diviser par $\pi^{ed}$ et ainsi construire l'application composée précédente,
\[ \zeta_\tau^d : \bigwedge^d D_{\sigma\tau} \fleche \bigwedge^d D_\tau\]
Finalement, on a multiplié $\bigwedge^d V_{\sigma \tau}$ par $\pi^{\min(d,d_{\tau,i})}$ pour tout $i$, puis divisé par $\pi^{ed}$, ce qui est bien ce qui était annoncé.
\end{proof}

Fixons un élément $\tau \in \T$, et $i$ un entier compris entre $1$ et $e$ tel que $d_{\tau,i} \neq 0$. Notre but est de construire une application $\bigwedge^{d_{\tau,i}} \Fil^{[i]}D_\tau \fleche \bigwedge^{d_{\tau,i}} \Fil^{[i]}D_\tau$ à partir de $V^f :\bigwedge^{d_{\tau,i}} \Fil^{[i]}D_\tau \fleche \bigwedge^{d_{\tau,i}} D_\tau$, et en divisant cette application par une puissance optimale de $\pi$, lorsque $d_{\tau,i}$ est non nul. Partant de $\bigwedge^{d_{\tau,i}} \Fil^{[i]}D_\tau$, plutôt que d'appliquer directement $V$, on peut composer les applications de descente dans la filtration, données par le corollaire \ref{corpi}, et donc considérer alors la flèche suivante,
\begin{center}
\begin{tikzpicture}[description/.style={fill=white,inner sep=2pt}] 
\matrix (m) [matrix of math nodes, row sep=3em, column sep=2.5em, text height=1.5ex, text depth=0.25ex] at (0,0)
{ 
\bigwedge^{d_{\tau,i}} \Fil^{[i]} D_\tau & & \bigwedge^{d_{\tau,i}} D_\tau\\
\bigwedge^{d_{\tau,i}} \Fil^{[i-1]} D_\tau & & \\
 & & \\
\bigwedge^{d_{\tau,i}} \Fil^{[1]} D_\tau & &\\
\bigwedge^{d_{\tau,i}} \Fil^{[0]} D_\tau  & & \\
 };

\path[->,font=\scriptsize] 
(m-1-1) edge node[auto,left] {$\pi^{\min(d_{\tau,i},d_{\tau,i})} = \pi^{d_{\tau,i}}$} (m-2-1)
(m-5-1) edge node[auto,right] {$\frac{1}{\pi^{ed_{\tau,i}}}$} (m-1-3)
(m-4-1) edge node[auto,left] {$\pi^{\min(d_{\tau,i},d_{\tau,1})}$} (m-5-1)
;
\path[dashed,->,font=\scriptsize]
(m-2-1) edge node[auto,left] {$\pi^{\min(d_{\tau,i},d_{\tau,2})}\circ\dots\circ\pi^{\min(d_{\tau,i},d_{\tau,i-1})}$} (m-4-1)
; 
\end{tikzpicture}
\end{center}
Cette composée correspond à une division par $\pi^{ed_{\tau,i} - \sum_{j \leq i} \min(d_{\tau,i},d_{\tau,j})}$, et nous noterons $Div_{\tau,i}$ cette application. Nous pouvons alors composer celle-ci avec les applications $\zeta_{\tau'}^{d_{\tau,i}}$ construites dans la proposition \ref{pro35}. Plus précisément, on peut étudier le morphisme 
\[(\bigwedge^{d_{\tau,i}} V_{\sigma \tau}) \circ \zeta_{\sigma\tau}^{d_{\tau,i}} \circ \dots \circ  \zeta_{\sigma^{-1}\tau}^{d_{\tau,i}} \circ Div_{\tau,i} : \bigwedge^{d_{\tau,i}} \Fil^{[i]} D_\tau \fleche \bigwedge^{d_{\tau,i}} \Fil^{[e]}D_\tau\]
Nous allons enfin utiliser la multiplication par une puissance de $\pi$ pour arriver dans $\bigwedge^{d_{\tau,i}} \Fil^{[i]} D_\tau$. Notons $Mul_{\tau,i}$ l'application composée
\begin{center}
\begin{tikzpicture}[description/.style={fill=white,inner sep=2pt}] 
\matrix (m) [matrix of math nodes, row sep=3em, column sep=2.5em, text height=1.5ex, text depth=0.25ex] at (0,0)
{ 
\bigwedge^{d_{\tau,i}} \Fil^{[e]}D_\tau \\
\bigwedge^{d_{\tau,i}} \Fil^{[e-1]} D_\tau \\
 \\
\bigwedge^{d_{\tau,i}} \Fil^{[i+1]} D_\tau\\
\bigwedge^{d_{\tau,i}} \Fil^{[i]} D_\tau \\
 };

\path[->,font=\scriptsize] 
(m-1-1) edge node[auto,left] {$\pi^{\min(d_{\tau,i},d_{\tau,e})}$} (m-2-1)

(m-4-1) edge node[auto,left] {$\pi^{\min(d_{\tau,i},d_{\tau,i+1})}$} (m-5-1)
;
\path[dashed,->,font=\scriptsize] 
(m-2-1) edge node[auto,right] {$\pi^{\min(d_{\tau,i},d_{\tau,i+2})} \circ \dots\circ\pi^{\min(d_{\tau,i},d_{\tau,e-1})}$} (m-4-1);
\end{tikzpicture}
\end{center}

Cette application correspond \`a une multiplication par $\pi^{\sum_{j > i} \min(d_{\tau,i},d_{\tau,j}) }$. Définissons \\
$HA_{\tau}^{[i]} :~\bigwedge^{d_{\tau,i}} \Fil^{[i]} D_\tau \to~\bigwedge^{d_{\tau,i}} \Fil^{[i]} D_\tau$ l'application composée égale \`a

$$Mul_{\tau,i} \circ (\bigwedge^{d_{\tau,i}} V_{\sigma \tau}) \circ \zeta_{\sigma\tau}^{d_{\tau,i}} \circ \dots \circ  \zeta_{\sigma^{-1}\tau}^{d_{\tau,i}} \circ Div_{\tau,i}$$




\begin{theo} 
\phantomsection\label{divparf}
Pour tout $\tau\in \T$ et $1 \leq i \leq e$, tels que $d_{\tau,i} \neq 0$, il existe une application $\sigma^{-f}$-linéaire,
\[ HA_{\tau}^{[i]} : \bigwedge^{d_{\tau,i}} \Fil^{[i]}D_\tau \fleche  \bigwedge^{d_{\tau,i}} \Fil^{[i]}D_\tau\]
telle que $\pi^{\sum_{j,\tau'} \max(d_{\tau,i} - d_{\tau',j},0)}HA_{\tau}^{[i]} = \bigwedge^{d_{\tau,i}}V^f$.
De plus, cette application passe au quotient en une application $\sigma^{-f}$-linéaire
\[\Ha_{\tau}^{[i]} :\det\Gr^{[i]}\omega_{G,\tau} \fleche \det \Gr^{[i]}\omega_{G,\tau}\]
\end{theo}

\begin{proof} 
Rappelons que pour $\tau' \in \T$, l'application $\zeta_{\tau'}^{d_{\tau,i}}$ est égale \`a $\bigwedge^{d_{\tau,i}} V_{\sigma \tau'}$ divisé par $\pi^{k_{\tau',d_{\tau,i}}}$, avec
$$k_{\tau',d_{\tau,i}} = \sum_{j=1}^e \max(d_{\tau,i} - d_{\tau',j},0)$$
L'application $Div_{\tau,i}$ correspond \`a une division par $\pi^{ed_{\tau,i} - \sum_{j \leq i} \min(d_{\tau,i},d_{\tau,j})}$, et $Mul_{\tau,i}$ \`a une multiplication par $\pi^{\sum_{j > i} \min(d_{\tau,i},d_{\tau,j}) }$. Au final, l'application $HA_{\tau}^{[i]}$ est égale \`a $\bigwedge^{d_{\tau,i}} V^f$ divisé par $\pi^K$, avec
$$K = \sum_{\tau' \neq \tau} \sum_{j=1}^e \max(d_{\tau,i} - d_{\tau',j},0) + e d_{\tau,i} - \sum_{j=1}^i \min(d_{\tau,i},d_{\tau,j}) - \sum_{j = i+1}^e \min(d_{\tau,i},d_{\tau,j}) 
= \sum_{\tau' \in \T} \sum_{j=1}^e \max(d_{\tau,i} - d_{\tau',j},0) $$
Soit maintenant 
\[x \dans \Ker(\bigwedge^{d_{\tau,i}}\Fil^{[i]}D_\tau \fleche \bigwedge^{d_{\tau,i}} \Gr^{[i]} \omega_{G,\tau}) = \im(\Fil^{[i-1]}D_\tau \otimes \bigwedge^{d_{\tau,i}-1} \Fil^{[i]}D_\tau \fleche \bigwedge^{d_{\tau,i}} \Fil^{[i]}D_\tau)\] 
La première application utilisée pour définir $Div_{\tau,i}$ est la multiplication par $\pi^{d_{\tau,i}} : \bigwedge^{d_{\tau,i}} \Fil^{[i]}D_\tau \to \bigwedge^{d_{\tau,i}} \Fil^{[i-1]}D_\tau$. L'image de $x$ par cette application est donc dans $\pi \cdot \bigwedge^{d_{\tau,i}} \Fil^{[i-1]}D_\tau$, et l'image de $x$ par $HA_\tau^{[i]}$ est dans $\pi \cdot \bigwedge^{d_{\tau,i}} \Fil^{[i]}D_\tau$. Puisque $\bigwedge^{d_{\tau,i}} \Gr^{[i]} \omega_{G,\tau}$ est tué par $\pi$, on en déduit que l'application $HA_\tau^{[i]}$ passe au quotient et définit une application $\sigma^{-f}$-linéaire
$$\Ha_\tau^{[i]} : \bigwedge^{d_{\tau,i}} \Gr^{[i]}\omega_{G,\tau} \fleche \bigwedge^{d_{\tau,i}} \Gr^{[i]}\omega_{G,\tau}$$
\end{proof}

L'application précédente est construite à l'aide de $V^f$, que l'on divise par une certaine puissance de $\pi$. Nous verrons dans la section $\ref{sect4}$ que cette division est optimale, c'est-à-dire qu'il existe des groupes $p$-divisibles (avec condition PR pour $\mu$) pour lesquels l'application $\Ha_\tau^{[i]}$ est non nulle. \\

\begin{defi}
Soit $\tau \in \mathcal T$ et $1 \leq i \leq e$ tels que $d_{\tau,i} = 0$, alors on pose par convention,
\[ \det \Gr^{[i]}\omega_{G,\tau} = k,\]
et définit $\Ha^{[i]}_\tau : \det \Gr^{[i]}\omega_{G,\tau}  \fleche \det \Gr^{[i]}\omega_{G,\tau}$ comme l'application $\sigma^{-f}$.
\end{defi}

L'application $\Ha_\tau^{[i]}$ induit un morphisme linéaire $\det \Gr^{[i]} \omega_{G,\tau} \to (\det \Gr^{[i]} \omega_{G,\tau})^{p^f}$, et donc une section du faisceau $(\det \Gr^{[i]} \omega_{G,\tau})^{p^f-1}$ sur $\Spec k$. Le produit de ces sections donne une section du faisceau $(\det \omega_G)^{p^f-1}$.

\begin{defi}
Soit $G$ un groupe $p$-divisible sur $k$ muni d'une action de $\mathcal O_L$. Le produit des sections construites précédemment permet de définir une section,
\[ {^\mu}\Ha(G) = \bigotimes_{\tau \in \mathcal T, 1 \leq i \leq e} \Ha^{[i]}_\tau \dans H^0(k,\det(\omega_G)^{p^f-1}),\]
appelée $\mu$-invariant de Hasse de $G$.
\end{defi}

\subsection{Cas général}

Soit $S$ un $k_L$-schéma, et soit $G \to S$ un groupe $p$-divisible tronqué d'échelon 1 (ou pour le dire plus rapidement, un $\mathcal{BT}_1$) de hauteur constante sur $S$. On suppose $G$ muni d'une action de $O_L$, et soit $\omega_G$ le faisceau conormal de $G$. Celui-ci est muni d'une action de $O_L$, et se décompose donc en $\omega_G = \oplus_{\tau \in \T} \omega_{G,\tau}$, avec $O_{L^{nr}}$ agissant par $\tau$ sur $\omega_{G,\tau}$. Les faisceaux $\omega_{G,\tau}$ sont localement libres. On rappelle que l'on s'est donné une collection d'entiers $\mu=(d_{\tau,i})_{\tau \in \T, 1 \leq i \leq e}$, et on supposera que $G$ vérifie la condition de Pappas-Rapoport suivante.

\begin{defi}
On dit que $G$ satisfait la condition de Pappas-Rapoport (PR en abrégé) pour la donnée $\mu$ s'il existe une filtration pour tout $\tau \in \T$
\[ 0=\omega_{G,\tau}^{[0]} \subset \omega_{G,\tau}^{[1]} \subset \omega_{G,\tau}^{[2]} \subset \dots \subset \omega_{G,\tau}^{[e]} = \omega_{G,\tau}\]
telle que
\begin{itemize}
\item $\omega_{G,\tau}^{[i]}$ est un sous-faisceau localement facteur direct de $\omega_{G,\tau}$ pour tout $1 \leq i \leq e$ et $\tau \in \T$.
\item $\pi \cdot \omega_{G,\tau}^{[i]} \subset \omega_{G,\tau}^{[i-1]}$ pour tout $1 \leq i \leq e$ et $\tau \in \T$.
\item si on note $\Gr^{[i]}\omega_{G,\tau} := \omega_{G,\tau}^{[i]}/\omega_{G,\tau}^{[i-1]}$, alors $\Gr^{[i]}\omega_{G,\tau}$ est localement libre de rang $d_{\tau,i}$ pour tout $1 \leq i \leq e$ et $\tau \in \T$.
\end{itemize}
\end{defi}

Regardons $\mathcal E$ l'évaluation du cristal de Berthelot-Breen-Messing (contravariant) de $G$ sur l'épaississement tautologique $S \overset{\id}{\fleche} S$ (\cite{BBM} partie $3.3$). Alors $\mathcal E$ est est un $\mathcal O_{S}$-module localement libre de rang constant muni d'une action de $O_L$. On suppose de plus que l'hypoth\`ese suivante est vérifiée.
\begin{IEEEeqnarray*}{l}
(BTO)\quad \mathcal E \text{ est un }\mathcal O_{S}\otimes_{\ZZ_p}O_L\text{-module localement libre}.
\end{IEEEeqnarray*}

\begin{rema} 
\phantomsection\label{remBTO1}
L'hypothèse (BTO) n'est pas automatique, le groupe $G$ pouvant par exemple être de $\pi$-torsion. Elle est cependant vérifiée si $G$ est la $p$-torsion d'un groupe $p$-divisible muni d'une action de $O_L$ (\cite{FGL} Lemme B.1.1).
\end{rema}

Lorsque l'on a $G_\infty$ un groupe $p$-divisible sur $\Spec k$, on étudiait dans la section précédente $D$ son cristal évalué sur $W(k)\rightarrow k$. Ici, le cristal $\mathcal E$ de $G = G_\infty[p]$ s'identifie à $D/pD$. On va alors généraliser la construction de la partie précédente dans le cas d'une base générale, et en considérant seulement un $\mathcal{BT}_1$. En particulier on en déduira que la construction de nos 
invariants pour un groupe $p$-divisible (qui vérifie donc en particulier la condition (BTO)) ne dépend que de sa $p$-torsion.
Notons comme dans la partie précédente,
\[ \mathcal E = \bigoplus_{\tau} \mathcal E_\tau\]
et on a la filtration de Hodge de $G$ (\cite{BBM} corollaire $3.3.5$)
\[ 0 \fleche \omega_G \fleche \mathcal E \fleche \omega_{G^D}^\vee \fleche 0\]
qui est compatible aux décompositions selon les plongements $\tau$. Soit $\tau \in \T$ ; on note pour $1 \leq i \leq e$
\[ \Fil^{[i]}\mathcal E_\tau := \omega_{G,\tau}^{[i]} \subset \mathcal{E}_\tau \]
et pour un $\mathcal O_S \otimes_{\ZZ_p} O_L$-module $\mathcal M$, on note $\mathcal M /\pi := \mathcal M \otimes_{O_L} (O_L/ \pi O_L)$.
La multiplication par $\pi \dans O_L$ induit une application,
\[ \pi : \Fil^{[i]}\mathcal E_\tau \fleche \Fil^{[i-1]}\mathcal E_\tau\]
qui passe alors au quotient en une application (non nécessairement triviale),
\[ \pi : \Fil^{[i]}\mathcal E_\tau/\pi \fleche \Fil^{[i-1]}\mathcal E_\tau/\pi\]
Puisque $\mathcal E$ est un $\mathcal O_S \otimes_{\ZZ_p} O_L$-module localement libre, on obtient facilement le lemme suivant.

\begin{lemm}
Soit $\tau \in \T$, et $0 \leq j \leq e-1$ un entier. Alors la multiplication par $\pi^j$ induit un isomorphisme
$$\pi^j : (\mathcal E_\tau)/\pi \to (\pi^j\mathcal E_\tau)/\pi$$
On notera $\frac{1}{\pi^j} : (\pi^j\mathcal E_\tau)/\pi \to (\mathcal E_\tau) / \pi$ l'isomorphisme inverse.
\end{lemm}

\begin{conv}
\phantomsection\label{conv}
Soit $S$ un schéma, et $\mathcal F$ un faisceau quasi-cohérent sur $S$. On pose alors pour la puissance extérieure nulle
\[\bigwedge^0 \mathcal F := \mathcal O_S\]
\end{conv}

\noindent Nous aurons besoin du lemme technique suivant. 

\begin{lemm}
\phantomsection\label{lem210}
Soit $S$ un schéma, $\mathcal G \subset \mathcal F$ deux $\mathcal O_S$-modules localement libre de rangs constants $n < f$, tels que $\mathcal G$ est localement facteur direct. On note $r = f -n$ le rang du quotient.
Soit $1 \leq d \leq n$, alors la flèche naturelle,
\[  q :\left(\bigwedge^{d}\mathcal G\right) \otimes \left(\bigwedge^{r} \mathcal F\right) \fleche \bigwedge^{d+r} \mathcal F\]
est surjective. De plus, son noyau est localement engendré par l'image des tenseurs
\[ (x\wedge x_1 \wedge \dots \wedge x_{d-1})\otimes(x \wedge y_1 \wedge \dots \wedge y_{r-1}), \quad x,x_i \dans \mathcal G, y_i \dans \mathcal F\]
\end{lemm}

\begin{proof}
Quitte à raisonner localement, supposons $S = \Spec R$, et $\mathcal F,\mathcal G, \mathcal F/\mathcal G$ libres, donnés respectivement par $M,N,M/N$ des $R$-modules libres. 
Soit $(e_1,\dots,e_n)$ une base de $N$ telle que $(e_1,\dots,e_{n+r})$ soit une base de $M$. Alors les éléments
$(e_{j_1}\wedge\dots\wedge e_{j_d}) \otimes (e_{l_1} \wedge \dots \wedge e_{l_r})$ où $j_1 < \dots < j_d \leq n$ et $l_1 < \dots < l_r$ induisent une base de $\bigwedge^dN\otimes\bigwedge^r M$. Une base de $\bigwedge^{d+r} M$ est constituée des éléments $e_{k_1} \wedge \dots \wedge e_{k_{d+r}}$, avec $1 \leq k_1 < \dots < k_{d+r} \leq n+r$. Mais pour un tel élément, on a $k_d \leq n$, et il est donc égal \`a $q((e_{k_1}\wedge\dots\wedge e_{k_d}) \otimes (e_{k_{d+1}} \wedge \dots \wedge e_{k_{d+r}}))$. 
Cela prouve que le morphisme $q$ est surjectif.

Soit 
\[ x = \sum_{\underline j \underline l} x_{\underline j,\underline l} (e_{j_1} \wedge \dots \wedge e_{j_d}) \otimes(e_{l_1}\wedge\dots\wedge e_{l_r})\]
décomposé dans la base précédente, tel que $x$ est dans le noyau de $q$. On peut supposer que le support des 
$\underline j,\underline l$ (c'est-\`a-dire l'ensemble $\{j_i;l_k : i \leq d, k \leq r\} \subset \{1,\dots,n+r\}$) qui apparaissent dans la somme est constant. En effet, on peut décomposer l'élément $x$ en $x = x_{J_1} +\dots + x_{J_s}$, o\`u $J_1, \dots, J_s$ sont des ensembles de $\{1,\dots,n+r\}$, et o\`u $x_{J_t}$ est une somme d'éléments de la base de $\bigwedge^d N \otimes_R \bigwedge^r M$ ayant leur support dans $J_t$ pour $1 \leq t \leq s$. Soit $t$ entre $1$ et $s$ ; si le cardinal de $J_t$ est strictement inférieur \`a $d+r$, alors $q(x_{J_t}) = 0$. Sinon, $J_t = \{ k_1 < \dots < k_{d+r} \}$, et définissons $y_t = e_{k_1} \wedge \dots \wedge e_{k_{d+r}} \in \bigwedge^{d+r} M$. Alors $q(x_{J_t})$ est dans le module engendré par $y_t$ ; comme les éléments $(y_u)_u$ ainsi définis forment une famille libre de $\bigwedge^{d+r} M$, on en déduit que $q(x_{J_t}) = 0$, pour tout $t$. \\
On suppose donc que le support des éléments $\underline j,\underline l$ apparaissant dans la somme définissant $x$ est constant. De plus, s'il existe des entiers $i$ et $k$ avec 
$j_i = l_k$, alors $(e_{j_1} \wedge \dots \wedge e_{j_d}) \otimes(e_{l_1}\wedge\dots\wedge e_{l_r})$ est dans le noyau de $q$, et est de la forme prédite par le lemme. Supposons donc que les $\underline j,\underline l$ ont tous le même support $S = \{k_1 < k_2 < \dots < k_{d+r}\}$, et on note $S_N := S \cap \{1,\dots n\} = \{k_1 < \dots < k_m\} $. Dans ce cas, on peut écrire $x$ sous la forme,
\[ x = \sum_{\sigma} x_{\sigma} (e_{\sigma(k_1)}\wedge\dots\wedge e_{\sigma(k_d)})\otimes (e_{\sigma(k_{d+1})} \wedge\dots\wedge e_{\sigma(k_{d+r})})\]
o\`u $\sigma$ parcourt les permutations de $S$ envoyant $\{1, \dots, d\}$ dans $S_N$. Quitte à regrouper les termes, on peut supposer que les $\sigma$ dans la somme soient l'identité sur $S \backslash S_N$, auquel cas la somme porte sur les permutations de $S_N$. Alors $x$ est dans le noyau de $q$ si et seulement si $\sum_\sigma \varepsilon(\sigma) x_\sigma= 0$.  
Notons
$$e_\sigma := (e_{\sigma(k_1)}\wedge\dots\wedge e_{\sigma(k_d)}) \otimes (e_{\sigma(k_{d+1})} \wedge\dots\wedge e_{\sigma(k_{d+r})})$$
o\`u $\sigma$ est une permutation de $S_N$ (prolongée en une permutation de $S$). Alors $\Ker q$ est engendré par les éléments $e_{\id} - \varepsilon(\sigma)e_\sigma$. Puisque l'ensemble des permutations est engendré par les transpositions, le noyau de $q$ est engendré par $e_\sigma + e_{\tau\sigma}$, o\`u $\sigma$ est une permutation de $S_N$, et $\tau$ une permutation de $S_N$. L'élément $\tau$ permute $\sigma(k_i)$ et $\sigma(k_j)$ pour des entiers $i < j$. Si $j \leq d$ ou $i > d$, l'élément $e_\tau + e_{\tau\sigma}$ est nul. On suppose donc $i \leq d < j$. Supposons pour simplifier les notations que $\sigma = \id$, $i=1$ et $j=d+1$ (on se ram\`ene facilement \`a ce cas) ; alors on peut écrire, 
\begin{eqnarray*}
e_\sigma + e_{\tau\sigma} & = (e_{k_1} \wedge \dots \wedge e_{k_d})\otimes(e_{k_{d+1}} \wedge \dots \wedge e_{k_{d+r}}) + (e_{k_{d+1}} \wedge e_{k_2} \wedge \dots \wedge e_{k_d})\otimes(e_{k_1} \wedge e_{k_{d+2}} \wedge \dots \wedge e_{k_{d+r}}) \\
& = ((e_{k_1} + e_{k_{d+1}}) \wedge e_{k_2} \wedge \dots \wedge e_{k_d})\otimes((e_{k_1} + e_{k_{d+1}}) \wedge e_{k_{d+2}} \wedge \dots \wedge e_{k_{d+r}}) - \\
& (e_{k_1} \wedge e_{k_2} \wedge \dots \wedge e_{k_d})\otimes(e_{k_1} \wedge e_{k_{d+2}} \wedge \dots \wedge e_{k_{d+r}}) - \\
& (e_{k_{d+1}} \wedge e_{k_2} \wedge \dots \wedge e_{k_d})\otimes(e_{k_{d+1}} \wedge e_{k_{d+2}} \wedge \dots \wedge e_{k_{d+r}})
\end{eqnarray*}
et ces 3 derniers termes sont de la forme annoncée.
\end{proof}

Ce lemme permet de prouver facilement la proposition suivante. Les produits extérieurs de faisceaux seront pris sur $\mathcal{O}_S$ (notons que nous ne définissons les produits extérieurs que pour des faisceaux avec une action de $O_L$ tués par $\pi$).

\begin{prop} 
\phantomsection\label{lem211}
Soit $\tau \in \T, 1 \leq i \leq e$ et $d \geq d_{\tau,i}$ un entier. Alors la flèche naturelle
\[ \left(\bigwedge^{d_{\tau,i}} \Fil^{[i]} \mathcal E_\tau/\pi\right) \otimes \left(\bigwedge^{d-d_{\tau,i}} \Fil^{[i-1]} \mathcal E_\tau/\pi\right) \fleche \bigwedge^d \Fil^{[i]}\mathcal E_\tau/\pi\]
est surjective. De plus, il existe une application $\bigwedge^{d_{\tau,i}}\pi : \bigwedge^d \Fil^{[i]}\mathcal E_\tau/\pi \to \bigwedge^d \Fil^{[i-1]}\mathcal E_\tau/\pi$, telle que le diagramme suivant commute

\begin{center}
\begin{tikzpicture}[description/.style={fill=white,inner sep=2pt}] 
\matrix (m) [matrix of math nodes, row sep=3em, column sep=2.5em, text height=1.5ex, text depth=0.25ex] at (0,0)
{ 
\left(\bigwedge^{d_{\tau,i}} \Fil^{[i]} \mathcal E_\tau/\pi\right) \otimes \left(\bigwedge^{d-d_{\tau,i}} \Fil^{[i-1]} \mathcal E_\tau/\pi\right) & & \bigwedge^d \Fil^{[i]}\mathcal E_\tau/\pi \\
\bigwedge^{d} \Fil^{[i-1]} \mathcal E_\tau/\pi& & \\
 };

\path[->,font=\scriptsize] 
(m-1-1) edge node[auto] {$$} (m-1-3)
(m-1-1) edge node[auto,left] {$\bigwedge^{d_{\tau,i}}\pi \otimes \bigwedge^{d-d_{\tau,i}}\id$} (m-2-1);
\path[dashed,->,font=\scriptsize] 
(m-1-3) edge node[auto] {$\bigwedge^{d_{\tau,i}} \pi $} (m-2-1)
;
\end{tikzpicture}
\end{center}
\end{prop}

\begin{rema}
Si $d_{\tau,i} = 0$, l'application $\bigwedge^{d_{\tau,i}}\pi$ précédente est l'identité ; en effet $\Fil^{[i]}\mathcal E_\tau = \Fil^{[i-1]}\mathcal E_\tau$. Si $d=d_{\tau,i}$, c'est simplement la puissance extérieur de l'application $\pi$.
\end{rema}

Lorsque $G$ est la $p$-torsion d'un groupe $p$-divisible sur un corps parfait, alors l'application $\bigwedge^{d_{\tau,i}}\pi$ précédente coïncide avec $\pi^{\min(d,d_{\tau,i})}$, donnée par le corollaire \ref{corpi}. On espère que l'abus de notation $\bigwedge^{d_{\tau,i}}\pi$ pour l'application induite ne causera pas de confusion.

\begin{proof} 
D'après le corollaire précédent, le morphisme 
$$\left(\bigwedge^{d_{\tau,i}} \Fil^{[i]} \mathcal E_\tau\right) \otimes \left(\bigwedge^{d-d_{\tau,i}} \Fil^{[i-1]} \mathcal E_\tau\right) \fleche \bigwedge^d \Fil^{[i]}\mathcal E_\tau$$
est surjectif, et on notera $K$ son noyau. On a alors la surjectivité en réduisant modulo $\pi$. Pour montrer que l'application de l'énoncé $\bigwedge^{d_{\tau,i}}\pi \otimes \bigwedge^{d-d_{\tau,i}}\id$ passe au quotient, il suffit de montrer que le noyau de l'application
\[ \bigwedge^{d_{\tau,i}}\pi \otimes \bigwedge^{d-d_{\tau,i}}\id : \left(\bigwedge^{d_{\tau,i}} \Fil^{[i]} \mathcal E_\tau\right) \otimes \left(\bigwedge^{d-d_{\tau,i}} \Fil^{[i-1]} \mathcal E_\tau\right) \fleche \bigwedge^d \Fil^{[i]}\mathcal E_\tau /\pi\] 
contient $K$.
Quitte à raisonner localement, soit $(x \wedge y_1 \wedge \dots \wedge y_{d_i - 1}) \otimes (x \wedge z_1 \wedge \dots \wedge z_{d-d_i-1})$ dans $K$, où 
$x, z_i \dans \Fil^{[i-1]}\mathcal E_\tau$ et $y_i \dans \Fil^{[i]}\mathcal E_\tau$.
Son image par l'application précédente est dans $\pi \cdot \bigwedge^d \Fil^{[i]}\mathcal E_\tau /\pi = 0$.
\end{proof}

Pour tout $\tau$, le cristal de $G$ est muni des applications suivantes,
\[ V_\tau : \mathcal E_{\tau} \fleche \mathcal E_{\sigma^{-1} \tau}^{(p)} \quad \text{et} \quad F_\tau : \mathcal E_{\sigma^{-1} \tau}^{(p)} \fleche \mathcal E_{\tau}\]
et comme $G$ est un $\mathcal{BT}_1$, on a $\Ker V_\tau = \im F_\tau$. On peut tirer en arrière les filtrations sur les faisceaux $(\mathcal E_\tau)_\tau$ :
on définit pour $\tau \in \T$ et $0 \leq i \leq e$ 
\[ \mathcal F_\tau^{[i]} := V_\tau^{-1}((\Fil^{[i]}\mathcal E_{\sigma^{-1} \tau})^{(p)})\]
et on en déduit une autre filtration,
\[ \Ker V_\tau = \mathcal F_\tau^{[0]} \subset \mathcal F_\tau^{[1]} \subset \dots \subset \mathcal F_\tau^{[e]} = \mathcal E_{\tau}\]
la relation $\mathcal F_\tau^{[e]} = \mathcal E_{\tau}$ provenant du fait que $(\Fil^{[e]}\mathcal E_{\sigma^{-1} \tau})^{(p)} = \im V_\tau$ (voir \cite{EV} partie $3.1$). 

\begin{prop}
Soit $\tau \in \mathcal{T}$ ; chaque sous-faisceau $\mathcal F_\tau^{[i]}$ de $\mathcal E_\tau$ est localement facteur direct pour $0 \leq i \leq e$. De plus, le gradué 
\[ \mathcal F_\tau^{[i]}/\mathcal F_\tau^{[i-1]},\]
est localement libre de rang $d_{\sigma^{-1}\tau,i}$, et l'application $\pi : \mathcal E_\tau \fleche \mathcal E_\tau$ envoie $\mathcal F_\tau^{[i]}$ dans $\mathcal F_\tau^{[i-1]}$ pour $1 \leq i \leq e$.
\end{prop}

\begin{proof}
On a $\mathcal F_\tau^{[0]} = \Ker V_\tau = \im F_\tau$, et on a donc des isomorphismes $\mathcal F_\tau^{[0]} \simeq (\mathcal E_\tau / \Fil^{[e]}\mathcal E_{\sigma^{-1} \tau})^{(p)}$ et $\mathcal{E}_\tau / \mathcal F_\tau^{[0]} \simeq (\Fil^{[e]}\mathcal E_{\sigma^{-1} \tau})^{(p)}$. Cela prouve que le faisceau $\mathcal F_\tau^{[0]}$ est localement un facteur direct. \\
Soit maintenant $i$ un entier compris entre $1$ et $e$. Alors, le morphisme $V_\tau$ induit un isomorphisme 
\[ V_\tau :\mathcal F_\tau^{[i]}/\mathcal F_\tau^{[i-1]} \overset{\sim}{\fleche} (\Fil^{[i]}\mathcal E_{\sigma^{-1}\tau} / \Fil^{[i-1]}\mathcal E_{\sigma^{-1}\tau})^{(p)}\]
ce qui prouve que le sous-faisceau $\mathcal F_\tau^{[i]}$ est localement un facteur direct, et donne le rang du quotient $\mathcal F_\tau^{[i]}/\mathcal F_\tau^{[i-1]}$. Enfin, $\pi$ commute à $V_\tau$, d'où $\pi \cdot (\mathcal F_\tau^{[i]}) \subset \mathcal F_\tau^{[i-1]}$. 
\end{proof}

Comme les faisceaux sont de caractéristique $p$, on ne peut plus faire la division par $p$ que l'on effectuait pour construire l'application 
$\zeta_\tau^d$ de la section précédente, et nous devons donc contourner ce problème.

\begin{prop} \phantomsection\label{appdiv}
Soit $\tau \in \T$ et $1 \leq i \leq e$ un entier. Pour $d > d_{\tau,i}$, il existe une application $\bigwedge^{d_{\tau,i}}\pi : \bigwedge^d \mathcal F^{[i]}_{\sigma \tau}/\pi \to \bigwedge^d \mathcal F^{[i-1]}_{\sigma \tau}/\pi$, telle que le diagramme suivant commute
\begin{center}
\begin{tikzpicture}[description/.style={fill=white,inner sep=2pt}] 
\matrix (m) [matrix of math nodes, row sep=3em, column sep=2.5em, text height=1.5ex, text depth=0.25ex] at (0,0)
{ 
\left(\bigwedge^{d_{\tau,i}} \mathcal F_{\sigma \tau}^{[i]}/\pi\right) \otimes \left(\bigwedge^{d-d_{\tau,i}} \mathcal F_{\sigma \tau}^{[i-1]}/\pi\right) & & \bigwedge^d \mathcal F^{[i]}_{\sigma \tau}/\pi \\
\bigwedge^{d} \mathcal F^{[i-1]}_{\sigma \tau}/\pi& & \\
 };

\path[->,font=\scriptsize] 
(m-1-1) edge node[auto] {$$} (m-1-3)
(m-1-1) edge node[auto,left] {$\bigwedge^{d_{\tau,i}}\pi \otimes \bigwedge^{d-d_{\tau,i}}\id$} (m-2-1);
\path[dashed,->,font=\scriptsize] 
(m-1-3) edge node[auto] {$\bigwedge^{d_{\tau,i}} \pi$} (m-2-1)
;
\end{tikzpicture}
\end{center}
Lorsque $d_{\tau,i} = 0$, cette application est simplement l'identité.
Pour $d \leq d_{\tau,i}$, l'application $\bigwedge^d \pi$ induit un morphisme
\[ \bigwedge^{d} \mathcal F_{\sigma \tau}^{[i]}/\pi \fleche \bigwedge^d \mathcal F_{\sigma \tau}^{[i-1]}/\pi.\]
\end{prop}

\begin{proof}
Les faisceaux $\mathcal F_{\sigma \tau}^{[i]}$ étant localement facteur direct, il suffit d'invoquer le lemme \ref{lem210}. La démonstration de la factorisation est la même que 
celle de la proposition \ref{lem211}.
\end{proof}

Pour que nos constructions coïncident (pas seulement à un inversible près), introduisons $u \dans O_L^\times$ tel que $p = u\pi^e$.

\begin{prop} 
Soit $\tau \in \T$. Pour $d > d_{\tau,1}$, l'application
\[ H_\tau^d  : \bigwedge^{d_{\tau,1}} \mathcal F^{[1]}_{\sigma \tau} \otimes \bigwedge^{d-d_{\tau,1}} \Ker V_{\sigma \tau} \fleche \bigwedge^d \mathcal E_{\tau}^{(p)}/\pi\]
induite localement, sous l'isomorphisme $\Ker V_{\sigma \tau} = \im F_{\sigma \tau}$, par
\[ x_1 \wedge \dots \wedge x_{d_{\tau,1}} \otimes F_{\sigma \tau} y_{1} \wedge \dots \wedge F_{\sigma \tau} y_{d-d_{\tau,1}} \longmapsto 
\frac{1}{\pi^{e-1}}V_{\sigma \tau} x_1 \wedge \dots \wedge \frac{1}{\pi^{e-1}}V_{\sigma \tau} x_{d_{\tau,1}} \wedge uy_{1} \wedge \dots \wedge uy_{d-d_{\tau,1}}\]
est bien définie. De plus, elle se factorise à travers la surjection $ \bigwedge^{d_{\tau,1}} \mathcal F^{[1]}_{\sigma \tau} \otimes \bigwedge^{d-d_{\tau,1}} \Ker V_{\sigma \tau} \to \bigwedge^d \mathcal F^{[1]}_{\sigma \tau}$, et induit une application
\[ H_\tau^d : \bigwedge^d \mathcal F^{[1]}_{\sigma \tau}/\pi \fleche \bigwedge^d \mathcal E_{\tau}^{(p)}/\pi\]
Si $d \leq d_{\tau,1}$, on définit $H_\tau^d : \bigwedge^d \mathcal F^{[1]}_{\sigma \tau}/\pi \fleche \bigwedge^d \mathcal E_{ \tau}^{(p)}/\pi$ par $H_\tau = \bigwedge^d ( \frac{1}{\pi^{e-1}} \circ V_{\sigma \tau} )$.
\end{prop}

\begin{proof} 
Démontrons la première assertion. Considérons la flèche,
\begin{equation*}
\bigwedge^{d_{\tau,1}} \mathcal F^{[1]}_{\sigma \tau} \otimes \bigwedge^{d-d_{\tau,1}} \Ker V_{\sigma \tau} \fleche \bigwedge^d \mathcal E_\tau[\pi]^{(p)}\end{equation*}
induite localement, sous l'isomorphisme $\Ker V_{\sigma \tau} = \im F_{\sigma \tau}$, par,
\[ x_1 \wedge \dots \wedge x_{d_{\tau,1}} \otimes F_{\sigma \tau} y_{1} \wedge \dots \wedge F_{\sigma \tau} y_{d-d_{\tau,1}} \longmapsto 
V_{\sigma \tau}x_1 \wedge \dots \wedge V_{\sigma \tau} x_{d_{\tau,1}} \wedge \pi^{e-1}uy_{1} \wedge \dots \wedge \pi^{e-1}uy_{d-d_{\tau,1}}\]
Cette application est bien définie car si $y_1,y_1' \dans \mathcal E_\tau^{(p)}$ vérifient que $F_{\sigma \tau} y_1 = F_{\sigma \tau} y_1'$ alors 
$y_1 - y_1' \dans \Ker F_{\sigma \tau} = \im V_{\sigma \tau} = \Fil^{[e]}\mathcal E_{ \tau}^{(p)}$
et donc $\pi^{e-1}u(y_1-y_1') \dans \Fil^{[1]}\mathcal E_{\tau}^{(p)}$. Or $V_{\sigma \tau} x_1,\dots V_{\sigma \tau} x_{d_{\tau,1}} \dans \Fil^{[1]}\mathcal E_\tau^{(p)}$, et comme ce dernier module est localement libre sur $\mathcal O_S$ de rang $d_{\tau,1}$, on a
$$V_\tau x_1 \wedge \dots \wedge V_\tau x_{d_{\tau,1}} \wedge u\pi^{e-1}(y_{1}-y_1') \wedge u\pi^{e-1} y_2 \wedge \dots \wedge u\pi^{e-1}y_{d-d_{\tau,1}} = 0 $$ 
ce qui prouve le résultat.
L'isomorphisme $\frac{1}{\pi^{e-1}} : \mathcal E_\tau[\pi]^{(p)} \simeq \mathcal E_\tau^{(p)}/\pi$ permet de définir l'application $H_\tau^d$. \\
D'après le lemme \ref{lem210}, on a une flèche surjective,
\[\bigwedge^{d_{\tau,1}} \mathcal F^{[1]}_{\sigma \tau} \otimes \bigwedge^{d-d_{\tau,i}} \Ker V_{\sigma \tau} \fleche \bigwedge^{d} \mathcal F^{[1]}_{\sigma \tau}\]
dont le noyau est localement engendré, d'après le lemme \ref{lem210}, par les tenseurs,
\[ x\wedge y_1 \wedge \dots \wedge y_{d_{\tau,1}-1} \otimes x \wedge F_{\sigma \tau} z_1 \wedge \dots \wedge F_{\sigma \tau} z_{d-d_{\tau,1}-1}\]
où $x \dans \Ker V_{\sigma \tau}$ donc $V_ {\sigma \tau} x = 0$, et donc l'application $H_\tau^d$ passe au quotient.
\end{proof}

\begin{defi} 
Pour tout $\tau \in \T$, et tout entier $d \geq 1$, on peut alors construire une application,
\[ \zeta_\tau^d : \bigwedge^d \mathcal E_{\sigma\tau}/\pi \fleche \bigwedge^d \mathcal E_\tau^{(p)}/\pi\]
en considérant la composée des applications précédentes
\[ \zeta_\tau^d = H_\tau^d \circ \bigwedge^{\min(d,d_{\tau,2})}\pi \circ \dots \circ \bigwedge^{\min(d,d_{\tau,e})} \pi.\]
\end{defi}

\begin{center}
\begin{tikzpicture}[description/.style={fill=white,inner sep=2pt}] 
\matrix (m) [matrix of math nodes, row sep=3em, column sep=2.5em, text height=1.5ex, text depth=0.25ex] at (0,0)
{ 
\bigwedge^{d} \mathcal E_{\sigma\tau}/\pi = \bigwedge^d \mathcal F_{\sigma \tau}^{[e]} /\pi& & & &  \bigwedge^{d} \mathcal E_\tau^{(p)}/\pi  \\
\bigwedge^d \mathcal F_{\sigma \tau}^{[e-1]} / \pi & & & &\\
 && & & \\
\bigwedge^d \mathcal F_{\sigma \tau}^{[1]} / \pi &  && & \\
 };

\path[->,font=\scriptsize] 
(m-1-1) edge node[auto,left] {$\bigwedge^{\min(d,d_{\tau,e})}\pi$} (m-2-1)
(m-4-1) edge node[auto,right] {$H_\tau^d$} (m-1-5)
(m-1-1) edge node[auto] {$\zeta_\tau^d$} (m-1-5);
\path[dashed,->,font=\scriptsize] 
(m-2-1) edge node[auto,left] {$\bigwedge^{\min(d,d_{\tau,2})}\pi\circ\dots\circ\bigwedge^{\min(d,d_{\tau,e-1})}\pi$} (m-4-1)
;
\end{tikzpicture}
\end{center}

Cette application est, lorsque $\mathcal E$ est le cristal de la $p$-torsion d'un groupe $p$-divisible $G_\infty$ sur un corps parfait, la même que l'application $\zeta_\tau^d$ définie dans la section 
précédente, ce qui justifie que l'on utilise la même notation. En effet, on peut construire aussi $H_\tau^d$ sur le cristal de $G_\infty$, et dans la section précédente pour construire $\zeta_\tau^d$, on appliquait 
$\bigwedge^d V_{\sigma \tau}$ suivi des applications $\bigwedge^{\min(d,d_{\tau,i})}\pi$, pour $i = e,\dots,1$, et enfin d'une division par $\pi^{de}$. Ici, on applique d'abord les morphismes $\bigwedge^{\min(d,d_{\tau,i})}\pi$ pour $i = e,\dots,2$, sur la filtration $\mathcal F_{\sigma\tau}^{[i]}$, puis l'application $H_\tau$. Dans le cas d'un groupe $p$-divisible sur un corps parfait, cette application est égale à $\frac{1}{\pi^{d_{\tau,1}(e-1)}}\bigwedge^{d_{\tau,1}} V_{\sigma\tau} \otimes \bigwedge^{d-d_{\tau,1}} u F_{\sigma\tau}^{-1}$ si $d > d_{\tau,1}$, et à $\frac{1}{\pi^{d(e-1)}}\bigwedge^{d} V_{\sigma\tau}$ sinon. Or comme sur l'isocristal d'un groupe $p$-divisible, $F^{-1} = \frac{1}{p}V$, $H_\tau$ est donc égal à $\frac{\pi^{\min(d,d_{\tau,1})}}{\pi^{de}}\bigwedge^d V_{\sigma \tau}$. 

\begin{defi}
Pour tout $\tau \in \T$ et $1 \leq i \leq e$, définissons le fibré en droites sur $S$
\[ \mathcal L_{\tau}^{[i]} := \det\left(\Gr^{[i]}\omega_{G,\tau}\right)\] 
\end{defi}

À partir de maintenant, faisons l'hypothèse simplificatrice suivante,
\begin{hypo} 
\phantomsection\label{hypdec}
Supposons que pour tout $\tau$ on a $d_{\tau,1} \geq d_{\tau,2} \geq \dots \geq d_{\tau,e}$.
\end{hypo}

\begin{prop} 
\phantomsection\label{prop215}
Soit $\tau \in \T$ et $1 \leq i \leq e$ ; en utilisant les applications précédentes, on peut construire une application
\[ HA_{\tau}^{[i]} : \bigwedge^{d_{\tau,i}} \Fil^{[i]}\mathcal E_\tau/\pi \fleche \bigwedge^{d_{\tau,i}}\Fil^{[i]}\mathcal E_\tau^{(p^f)}/\pi\]
qui correspond dans le cas où $S = \Spec k$ est un corps parfait de caractéristique $p$, et $G$ provient d'un groupe $p$-divisible, à l'application de la section précédente.
De plus, cette application passe au quotient, et induit une application
\[ \Ha_{\tau}^{[i]} : \mathcal L_{\tau}^{[i]} \fleche\left(\mathcal L_{\tau}^{[i]})\right)^{p^f}\]
De manière équivalente, on a une section d'un fibré inversible, \[\Ha_{\tau}^{[i]} \dans H^0 \left(S,\left(\mathcal L_{\tau}^{[i]}\right)^{p^f-1} \right)\]
Lorsque $d_{\tau,i} = 0$, sous la convention \ref{conv}, on pose $\Ha_\tau^{[i]} = 1$.
\end{prop}

\begin{proof} 
On peut regarder la composée suivante
\begin{center}
\begin{tikzpicture}[description/.style={fill=white,inner sep=2pt}] 
\matrix (m) [matrix of math nodes, row sep=3em, column sep=2.5em, text height=1.5ex, text depth=0.25ex] at (0,0)
{ 
\bigwedge^{d_{\tau,i}} \mathcal E_\tau/\pi & & & &  & & \bigwedge^{d_{\tau,i}} \Fil^{[e]}\mathcal E_\tau^{(p^f)}/\pi  \\
\bigwedge^{d_{\tau,i}} (\pi^{e-i}\mathcal{E}_\tau)/\pi&  & && & &\bigwedge^{d_{\tau,i}} \Fil^{[i]}\mathcal E_\tau^{(p^f)}/\pi\\
 };

\path[->,font=\scriptsize] 
(m-2-1) edge node[auto] {$\bigwedge^{d_{\tau,i}}\frac{1}{\pi^{e-i}}$} (m-1-1)
(m-1-1) edge node[auto] {$(\bigwedge^{d_{\tau,i}}V_{\sigma \tau})^{(p^{f-1})} \circ (\zeta_{\sigma\tau}^{d_{\tau,i}})^{(p^{f-2})} \circ \dots \circ \zeta_{\sigma^{-1} \tau}^{d_{\tau,i}}$} (m-1-7)
(m-1-7) edge node[auto] {$(\bigwedge^{d_{\tau,i+1}}\pi \circ\dots\circ \bigwedge^{d_{\tau,e}}\pi)^{(p^f)}$} (m-2-7);

;
\end{tikzpicture}
\end{center}
La composée de cette application avec le morphisme $\bigwedge^{d_{\tau,i}} \Fil^{[i]}\mathcal E_\tau /\pi \to \bigwedge^{d_{\tau,i}} (\pi^{e-i}\mathcal{E}_\tau)/\pi$ donne l'application $HA_{\tau}^{[i]}$. \\
Cette application passe bien au quotient comme annoncé puisque l'image de $\Fil^{[i-1]} \mathcal E_\tau$ par l'application $\Fil^{[i]}\mathcal E_\tau \fleche (\pi^{e-i} \mathcal{E}_\tau) /\pi$ est nulle $(\Fil^{[i-1]} \mathcal E_\tau \subset \pi^{e-i+1}\mathcal E_\tau$).

Pour voir que cette application coïncide avec celle de la section précédente, lorsque $G$ est la $p$-torsion d'un groupe $p$-divisible $G_\infty$ sur un corps parfait, il suffit de regarder le cristal $D$ de $G_\infty$, qui est un $W_{O_L}(k)$-module libre. On a alors justifié que les applications $\bigwedge^{\min(d_{\tau,i},d_{\tau',j})}\pi$ et $\pi^{\min(d_{\tau,i},d_{\tau',j})}$ coïncidaient, ainsi que les applications 
$\zeta_{\tau'}^{d_{i,\tau}}$. Il faut alors voir que l'application 
\[\bigwedge^{d_{\tau,i}}\Fil^{[i]}\mathcal E_\tau /\pi \to \bigwedge^{d_{\tau,i}} (\pi^{e-i}\mathcal{E}_\tau)/\pi \overset{\bigwedge^{d_{\tau,i}}\frac{1}{\pi^{e-i}}}{\fleche}\bigwedge^{d_{\tau,i}}
(\mathcal{E}_\tau)/\pi,\]
coïncide avec l'application $Div_{\tau,i}$ de la section $\ref{corps}$. Rappelons que l'application $Div_{\tau,i}$ 
%
correspond à une division par $\pi$ à la puissance $ed_{\tau,i} - \sum_{j \leq i} \min(d_{\tau,i},d_{\tau,j})$. Mais cette quantité, sous l'hypothèse \ref{hypdec}, est égale à $ed_{\tau,i} - \sum_{j\leq i} d_{\tau,i} = (e-i)d_{\tau,i}$. L'application $Div_{\tau,i}$ est alors égale à $\bigwedge^{d_{\tau,i}}\frac{1}{\pi^{e-i}}$.
\end{proof}

\begin{rema} 
Si l'hypothèse \ref{hypdec} n'est pas vérifiée la construction a toujours un sens, mais certaines applications construites seront toujours $0$ (plus précisément $\bigwedge^{d_{\tau,i}}\frac{1}{\pi^{e-i}}$ pour un certain $(\tau,i)$ est alors nulle modulo $\pi$). Il devrait être possible de construire des invariants non identiquement nuls dans ce cas, mais l'hypothèse \ref{hypdec} simplifie la situation. Remarquons que les variétés définies par Pappas et Rapoport dans \cite{P-R} nécessitent de choisir un ordre sur les éléments $(d_{\tau,i})_i$ pour tout $\tau \in \T$, ce qui rend l'hypothèse \ref{hypdec} inoffensive en pratique.
\end{rema}

\noindent On peut alors construire l'invariant de Hasse $\mu$-ordinaire.

\begin{defi}
Soit $S$ un schéma de caractéristique $p$, et $G \to S$ un $\mathcal{BT}_1$ muni d'une action de $O_L$ qui vérifie l'hypothèse (BTO). On définit alors son $\mu$-invariant de Hasse comme le produit des invariants précédents
\[ {^\mu}\Ha = \bigotimes_{\tau,i} \Ha_\tau^{[i]} \dans H^0 \left(S,\det \left(\omega_{G}\right)^{p^f-1} \right).\]
\end{defi}

\section{Groupes $p$-divisibles $\mu$-ordinaires}
\phantomsection\label{sect4}

Dans cette partie, on fixe une donnée $\mu = (d_{\tau,i})_{\tau \in \T, 1 \leq i \leq e}$ d'entiers avec $0 \leq d_{\tau,i} \leq h$ pour $\tau \in \T$ et $1 \leq i \leq e$. On suppose également qu'ils sont ordonnés, c'est-à-dire que pour tout $\tau \in \T$
$$d_{\tau,1} \geq \dots \geq d_{\tau,e}$$

\subsection{Définition}
\phantomsection\label{sect21}
Soit $k$ un corps parfait de caractéristique $p$ contenant le corps résiduel de $L$, et soit $G$ un groupe $p$-divisible sur $k$. Soit $(M,F,V)$ le module de Dieudonné (contravariant) associé à $G$. On suppose que $G$ a une action de $O_L$. Le module $M$ muni du Frobenius est donc un $F$-cristal avec une action de $O_L$. Si $h$ est le rang de $M$ (en tenant compte de l'action de $O_L$), alors la hauteur de $G$ est $ef h$. \\
On suppose que $G$ satisfait la condition de Pappas-Rapoport pour la donnée $\mu$. Cette condition peut être vue sur le cristal $(M,F)$. \\
On définit les polygones $\Newt_{O_L} (G)$ et $\Hdg_{O_L} (G)$ comme étant égaux respectivement à $\Newt_{O_L} (M,F)$, $\Hdg_{O_L} (M,F)$. On rappelle que l'on a défini un polygone $\PR(\mu) := \PR(d_\bullet)$, et que l'on a les égalités
$$\Newt_{O_L} (G) \geq \Hdg_{O_L} (G) \geq \PR(\mu)$$

\begin{defi}
On dit que le groupe $p$-divisible $G$ satisfait la condition de Rapoport généralisée (pour la donnée $\mu$) si on a l'égalité
$$\Hdg_{O_L} (G) = \PR(\mu)$$
\end{defi}

Cette condition est automatique si $L$ est non ramifié. Lorsque la donnée $\mu$ est ordinaire (i.e. lorsque $d_{\tau,i}=d$ pour tout $\tau \in \T$ et $1 \leq i \leq e$), cette condition est équivalente au fait que le module $\omega_G$ soit libre sur $k \otimes_{\mathbb{Z}_p} O_L$. Cette dernière condition est souvent appelée condition de Rapoport, ce qui justifie la terminologie. \\
La condition de Rapoport généralisée signifie que la structure de $\omega_G$ comme $k \otimes_{\mathbb{Z}_p} O_L$-module est la meilleure possible étant donnée la condition PR. Par ailleurs, si $G$ satisfait la condition de Rapoport généralisée, alors la filtration sur chacun des $\omega_{G,\tau}$ dans la définition de la condition PR est uniquement déterminée.

\begin{defi}
On dit que le groupe $p$-divisible est $\mu$-ordinaire si on a l'égalité entre polygones
$$\Newt_{O_L} (G) = \PR(\mu)$$
\end{defi}

D'après les inégalités entre les trois polygones, si $G$ est $\mu$-ordinaire, alors il satisfait la condition de Rapoport généralisée. Nous donnerons dans la section suivante plusieurs caractérisations du fait d'être $\mu$-ordinaire. Avant cela, nous allons définir explicitement un groupe $p$-divisible $\mu$-ordinaire. \\
Soient $\beta = (\beta_{\tau})_{\tau \in \T}$ des entiers avec $0 \leq \beta_\tau \leq e$ pour tout $\tau \in \T$. On définit le module de Dieudonné $(M_{\beta},F,V)$ par
$M_\beta = \bigoplus_{\tau \in \T} M_{\beta,\tau}$ avec $M_{\beta,\tau} = W_{O_L,\tau} (k)$ pour tout $\tau \in \T$. Le Frobenius $F_\tau : M_{\beta, \sigma^{-1} \tau } \to M_{\beta, \tau}$ et le Verschiebung $V_\tau : M_{\beta,\tau} \to M_{\beta, \sigma^{-1} \tau}$ sont définis par 
$$F_\tau (x) = \pi^{\beta_\tau} \sigma(x)  \text{    et    }     V_\tau (y) = p \pi^{- \beta_\tau} \sigma^{-1} (y)$$
pour tout $\tau \in \T$, $x \in M_{\beta,\sigma^{-1} \tau}$ et $y \in M_{\beta,\tau}$. Remarquons que la valuation de $p \pi^{- \beta_\tau}$ est égale \`a $1 - \beta_\tau/e \geq 0$. \\
Le $F$-cristal $(M_\beta,F)$ a donc une action de $O_L$, et est de rang $1$. Les polygones $\Newt_{O_L} (M_\beta,F)$ et $\Hdg_{O_L} (M_\beta,F)$ sont égaux, et ont une seule pente égale \`a $(\sum_{\tau \in \T} \beta_\tau)/(ef)$. On notera $X_\beta$ le groupe $p$-divisible associé au module de Dieudonné $(M_\beta, F, V)$. C'est un groupe $p$-divisible sur $k$ de hauteur $ef$ muni d'une action de $O_L$. \\
Soit $r$ le cardinal de $\{ d_{\tau,i}, \tau \in \T, 1 \leq i \leq e \} \cap [1,h-1]$. On note
$$0 < D_1 < \dots < D_r <h$$
les éléments de cet ensemble. On pose également $D_0 = 0$ et $D_{r+1} = h$. Soit $1 \leq j \leq r+1$ ; pour $\tau \in \T$, on définit $\alpha_{j,\tau}$ comme le cardinal de l'ensemble $\{ 1 \leq i \leq e, d_{\tau,i} \geq D_j \}$. On définit $\alpha_j = (\alpha_{j,\tau})_{\tau \in \T}$ ; puisque les entiers $\alpha_{j,\tau}$ sont compris entre $0$ et $e$, on dispose donc du groupe $p$-divisible $X_{\alpha_j}$ sur $k$. Remarquons que $\alpha_{j+1} \leq \alpha_{j}$ pour tout $1 \leq j \leq r$.

\begin{defi}
\phantomsection\label{def33}
On définit le groupe $p$-divisible $X^{ord}$ sur $k$ par
$$X^{ord} := \prod_{j=1}^{r+1} X_{\alpha_j}^{D_j - D_{j-1}}$$
\end{defi}

\noindent Le groupe $p$-divisible $X^{ord}$ est muni d'une action de $O_L$, et est de hauteur $ef h$. De plus, on peut vérifier que les polygones $\Newt_{O_L} (X^{ord})$ et $\Hdg_{O_L} (X^{ord})$ sont égaux \`a $\PR(\mu)$.

\begin{prop}
Les polygones $\Newt_{O_L} (X^{ord})$ et $\Hdg_{O_L} (X^{ord})$ sont égaux \`a $\PR(\mu)$.
\end{prop}

\begin{proof} 
Soit $\tau \in \T$ ; le polygone $\Hdg_{O_L,\tau} (X^{ord})$ a pour pentes $\alpha_{j,\tau}/e$ avec multiplicité $D_j - D_{j-1}$ pour tout $1 \leq j \leq r+1$. Or pour tout $0 \leq j \leq r+1$, il y a exactement $\alpha_{j,\tau} - \alpha_{j+1,\tau}$ éléments parmi les $(d_{\tau,i})_{1 \leq i \leq e}$ qui sont égaux \`a $D_j$ (on pose $\alpha_{0,\tau} =e$ et $\alpha_{r+2,\tau} = 0$). On peut donc écrire pour $0 \leq i \leq h$
$$\PR_\tau(d_\bullet)(i) = \sum_{j=0}^{r+1} \frac{\alpha_{j,\tau} - \alpha_{j+1,\tau}}{e} \max(-h + D_j + i, 0)$$
On voit donc que les pentes du polygone $\PR_\tau(d_\bullet)$ sont exactement $\alpha_{j,\tau}/e$ avec multiplicité $D_j - D_{j-1}$ pour $1 \leq j \leq r+1$. Cela prouve que les polygones $\Hdg_{O_L,\tau} (X^{ord})$ et $\PR_\tau (d_\bullet)$ sont égaux pour tout $\tau \in \mathcal{T}$, ce qui donne l'égalité $\Hdg_{O_L}(X^{ord}) = \PR(\mu)$.
En particulier, comme les $\alpha_j$ sont ordonnés, $\Hdg_{O_L}(X^{ord})$ est la concatenation des $\Hdg_{O_L}(X_{\alpha_j}^{D_j-D_{j-1}})$. Or c'est aussi automatiquement le cas pour le polygone de Newton de $X^{ord}$, et comme les polygones de Hodge et Newton des $X_{\alpha_j}^{D_j-D_{j-1}}$ sont isoclins donc égaux, on en déduit l'égalité $\Newt_{O_L}(X^{ord}) = \Hdg_{O_L}(X^{ord})$.
\end{proof}

\subsection{Caractérisations}

Dans cette partie, nous allons donner plusieurs caractérisations des groupes $p$-divisibles $\mu$-ordinaires.

\begin{theo}
\phantomsection\label{equimuord}
Supposons que $k$ est algébriquement clos, et soit $G$ un groupe $p$-divisible sur $k$ avec une action de $O_L$ satisfaisant la condition PR pour la donnée $\mu$. Alors les conditions suivantes sont équivalentes.
\begin{itemize}
\item $G$ est $\mu$-ordinaire.
\item $G$ est isog\`ene \`a $X^{ord}$.
\item $G$ est isomorphe \`a $X^{ord}$.
\end{itemize}
\end{theo}

\noindent La condition $G$ isogène à $X^{ord}$ signifie qu'il existe une isogénie $O_L$-linéaire entre les groupes $p$-divisibles $G$ et $X^{ord}$. 

\begin{proof}
La troisième condition implique évidemment la deuxième. Puisque le polygone de Newton est invariant par isogénie, la deuxième condition implique la première. Supposons maintenant que $G$ est $\mu$-ordinaire, et soit $(M,F)$ le $F$-cristal associé à $G$. D'après le théorème de décomposition Hodge-Newton, on a une décomposition
$$M := \bigoplus_{j=1}^{r+1} M_j$$
avec $(M_j,F_j)$ un $F$-cristal avec action de $O_L$ de rang $D_j - D_{j-1}$. De plus, les polygones $\Hdg_{O_L,\tau} (M_j,F_j)$ et $\Newt_{O_L}(M_j,F_j)$ sont isoclines de pentes respectives $\alpha_{\tau,j}/e$ et $\sum_{\tau \in \T} \alpha_{\tau,j}/(ef)$. \\
Nous voulons démontrer que $G$ est isomorphe à $X^{ord}$. Comme le foncteur induit par le module de Dieudonné est pleinement fidèle, il suffit de montrer que $M$ est isomorphe au module de Dieudonné de $X^{ord}$, ou encore que chaque cristal $(M_j,F_j)$ est isomorphe à $(M_{\alpha_j}^{D_j - D_{j-1}},F)$. Le module $M_j$ se décompose en $\oplus_{\tau \in \T} M_{j,\tau}$, et le Frobenius $F_j$ induit des applications $F_{j,\tau} : M_{j,\sigma^{-1} \tau} \to M_\tau$. Comme le polygone $\Hdg_{O_L,\tau} (M_j,F_j)$ est isocline de pente $\alpha_{j,\tau}/e$, $F_\tau$ est divisible par $\pi^{\alpha_{j,\tau}}$. En divisant chaque $F_\tau$ par $\pi^{\alpha_{j,\tau}}$, on obtient un cristal $(M_j,F_j')$ avec action de $O_L$ dont le polygone de Newton est isocline de pente $0$. On en déduit que $(M_j,F_j') \simeq (M_{\beta_0}^{D_j-D_{j-1}},F)$, avec $\beta_0=(0,\dots,0)$, et donc que $(M_j,F_j) \simeq (M_{\alpha_j}^{D_j - D_{j-1}},F)$ pour tout $1 \leq j \leq r+1$.
\end{proof}

Le fait d'être $\mu$-ordinaire peut également être vu à l'aide des invariants de Hasse construits dans la section précédente. Dans le reste de cette partie, on considère un groupe $p$-divisible $G$ sur un corps algébriquement clos $k$, muni d'une action de $O_L$ et satisfaisant la condition PR pour la donnée $\mu$. On rappelle que les points de rupture du polygone $\PR(\mu)$ ont pour abscisses les éléments $h - d_{\tau,i}$, pour $\tau \in \T$ et $1 \leq i \leq e$.


\begin{prop}
\phantomsection\label{pro38}
Soient $\tau \in \T$ et $1 \leq i \leq e$ tel que $d_{\tau,i}$ ne soit pas égal à $0$ ou $h$. Alors
$$\Ha_{\tau}^{[i]} (G) \text{ est non nul } \Leftrightarrow \Newt_{O_L} (G) (h - d_{\tau,i}) = \PR(\mu) (h - d_{\tau,i})$$
\end{prop}

\begin{proof}
Soit $(M,F,V)$ le module de Dieudonné, $(D,F,V)$ le cristal de Dieudonné sur $W(k)$ associé à $G$. Comme remarqué après la définition \ref{defPRG},
\[ M/FM = \omega_G = VD/pD\]
et $V : M \fleche D$ est une isogénie, donc préserve les polygones de Newton. Puisque nous utilisons le Verschiebung $V$ pour construire les invariants de Hasse, nous allons relier le polygone de Newton de $G$ au polygone du $\sigma^{-1}$-cristal $(D,V)$. Les relations $FV = VF = p$, impliquent
$$\Newt_{O_L} (D,V) (x) = \Newt_{O_L} (G) (h-x) + x - \Newt_{O_L} (G) (h) $$
Notons que 
$$\PR(\mu) (h - d_{\tau,i}) = \frac{1}{ef}  \sum_{\tau' \in \T} \sum_{j=1}^e \max(d_{\tau',j} - d_{\tau,i},0)$$
La condition $\Newt_{O_L} (G) (h - d_{\tau,i}) = \PR(\mu) (h - d_{\tau,i})$ est donc équivalente à $\Newt_{O_L} (D,V) (d_{\tau,i}) = k_{\tau,i} / (ef)$, avec
$$k_{\tau,i} := \sum_{\tau' \in \T} \sum_{j=1}^e \max(d_{\tau,i} - d_{\tau',j},0)$$
Il suffit donc de prouver que $\Ha^{[i]}_\tau$ est inversible si et seulement si $\Newt_{O_L}(D,V)(d_{\tau,i}) = k_{\tau,i}/(ef)$. Remarquons que $k_{\tau,i}$ est exactement la puissance de $\pi$ par laquelle on divise le Verschiebung à la puissance $f$ pour construire $\Ha^{[i]}_\tau$. \\
Regardons alors le cristal 
$(\bigwedge^{d_{\tau,i}} D_\tau, V^f)$, le produit extérieur étant pris sur $W_{O_L,\tau}(k)$. Il suffit de voir à quelle condition le polygone de Newton de $(\bigwedge^{d_{\tau,i}} D_\tau, V^f)$ a pour valeur $k_{\tau,i}/e$ en 1, 
d'après la proposition \ref{ext} (plus exactement son analogue pour $(D,V)$). 
Pour cela regardons le réseau $(\bigwedge^{d_{\tau,i}} \Fil^{[i]} D_\tau, V^f)$, c'est bien un réseau puisque $pD_\tau \subset \Fil^{[i]}D_\tau \subset VD_{\sigma\tau}$. Or sur ce réseau $V^f$ est divisible par $\pi^{k_{\tau,i}}$ d'après le théorème \ref{divparf}, on peut donc regarder le cristal
\[ (\bigwedge^{d_{\tau,i}}\Fil^{[i]}D_\tau,\frac{1}{\pi^{k_{\tau,i}}}V^f ) =: (\Lambda, \phi)\]
Il suffit donc de voir à quelle condition le polygone de Newton du cristal $(\Lambda,\phi)$ (qui est  celui de $(\bigwedge^{d_{\tau,i}}D_\tau,V^f)$, dont on a soustrait $k_{\tau,i}/e$ à toutes les
pentes) a une pente 0, c'est à dire qu'il existe un sous-cristal sur lequel $\phi$ est inversible.
Mais si on regarde la projection
\[ \bigwedge^{d_{\tau,i}}\Fil^{[i]}D_\tau \overset{q}{\fleche} \bigwedge^{d_{\tau,i}} \omega^{[i]}_\tau/\omega^{[i-1]}_\tau\]
son noyau $N = \Ker q = \im(\Fil^{[i-1]}D_\tau\otimes \bigwedge^{d_i-1}\Fil^{[i]}D_\tau \fleche \bigwedge^{d_i}\Fil^{[i]}D_\tau)$ vérifie (c'est ce qu'on a vérifié pour voir que 
$\Ha_\tau^{[i]}$ passe au quotient) $\phi(N) \subset \pi \Lambda$.
Maintenant on peut écrire $\Lambda = \Lambda^{et} \oplus \Lambda^{nilp}$ (partie où $\phi$ est inversible, et partie où $\phi$ est nilpotent), et on a donc montré que 
\[ N \subset \Lambda^{nilp} + \pi \Lambda\]
Mais maintenant comme $\bigwedge^{d_{\tau,i}} \omega_\tau^{[i]}/\omega_\tau^{[i-1]}$ est de dimension 1 sur $k$, il est équivalent de demander qu'il existe un sous espace 
non nul sur lequel $\phi$ est inversible ou de demander que $\Ha_\tau^{[i]}$ soit inversible. On a donc
\[ \Ha_\tau^{[i]} \text{ inversible } \Leftrightarrow \Lambda^{et} \neq 0 \Leftrightarrow \Newt_{O_L,\tau}(\bigwedge^{d_{\tau,i}}D_\tau, V^f)(1) = k_{\tau,i}/e \Leftrightarrow \Newt_{O_L}(G)(h-d_{\tau,i}) = PR(\mu)(h-d_{\tau,i}).\qedhere\]
\end{proof}

\begin{coro} 
Le groupe $p$-divisible $G$ est $\mu$-ordinaire si et seulement si ${^\mu}\Ha(G)$ est non nul.
\end{coro}

\begin{coro}
Le groupe $p$-divisible $G$ est $\mu$-ordinaire si et seulement si $G[p]$ est isomorphe à $X^{ord}[p]$.
\end{coro}

\begin{proof}
On a déjà montré dans le théorème \ref{equimuord} que si $G$ est $\mu$-ordinaire, alors il est isomorphe à $X^{ord}$, donc a fortiori $G[p] \simeq X^{ord}[p]$.
Supposons donc que $G[p] \simeq X^{ord}[p]$. Alors ${^\mu}\Ha(G) = {^\mu}\Ha(X^{ord})$, puisque l'invariant ${^\mu}\Ha$ ne dépend que de la $p$-torsion, donc ${^\mu}\Ha(G)$ est inversible, 
et par le corollaire précédent $G$ est $\mu$-ordinaire.
\end{proof}

\subsection{Groupes $p$-divisibles $\mu$-ordinaires sur un anneau artinien}

Dans cette partie, nous étudions les groupes $p$-divisibles sur un anneau artinien de caractéristique $p$. Soit $k$ un corps algébriquement clos de caractéristique $p$, et soit $R$ une $k$-algèbre locale artinienne de corps résiduel $k$. Soit $G$ un groupe $p$-divisible sur $R$ muni d'une action de $O_L$. On suppose que $G$ satisfait la condition PR pour la donnée $\mu$.

\begin{defi}
On dit que $G$ est $\mu$-ordinaire si $G \times_{R} k$ l'est.
\end{defi}

\begin{prop} \phantomsection\label{filt}
On suppose que $G$ est $\mu$-ordinaire. Alors il existe une unique filtration de $G$ par des groupes $p$-divisibles
$$0=G_0 \subset G_1 \subset \dots \subset G_r \subset G_{r+1}=G  $$
telle que $ (G_j / G_{j-1}) \times_{R} k \simeq X_{\alpha_j}^{D_j - D_{j-1}}$ pour $1 \leq j \leq r+1$.
\end{prop}

\begin{proof}
Dans le cas non-ramifié, ce résultat est la proposition 2.1.9 de \cite{Moo}, et est prouvé en utilisant les anneaux de déformations universels. Nous allons plutôt utiliser la théorie de Grothendieck-Messing. Rappelons le résultat qui nous interresse,
\begin{theo}[Grothendieck-Messing, \cite{Mess} Theorem V.1.6]
Soit $S$ un schéma, $\mathcal I$ un faisceau d'ideaux localement nilpotents muni de puissances divisées qui définit un sous-schéma fermé $S_0$.
Soit $\mathcal{BT}(S)$ la catégorie des groupes $p$-divisibles sur $S$ et $\mathcal C(S)$ la catégorie formée des couples $(G_0,\Fil)$ où $G_0$ est un groupe $p$-divisible sur $S_0$ et $\Fil \subset \mathbb D(G_0)_{(S_0 \hookrightarrow S)}$ est une filtration admissible (i.e. localement facteur direct et qui relève $\omega_{G_0} \subset \mathbb D(G_0)_{(S_0 \to S_0)}$). Alors le foncteur,
\[ 
\begin{array}{ccc}
 \mathcal{BT}(S)  & \fleche  & \mathcal C(S)  \\
G  & \longmapsto   &   (G\times_S S_0,\omega_G \subset \mathbb D(G_0)_{(S_0 \hookrightarrow S)})
\end{array}
\]
est une anti-équivalence de catégories.
\end{theo}

Démontrons maintenant la proposition \ref{filt}. Soit $I$ l'unique idéal maximal de $R$, et supposons dans un premier temps que $I$ est de carré nul. Cet idéal est en particulier muni de puissances divisées, et soit $\mathcal{D}$ le cristal de Dieudonné (contravariant) de $G \times_{R} k$ évalué sur l'épaississement $R \to k$. Alors $\mathcal{D}$ est un $R \otimes_{\mathbb{Z}_p} O_L$-module libre de rang $h$ ; de plus comme $G$ est $\mu$-ordinaire, on a 
$$\mathcal{D} = \bigoplus_{j=1}^{r+1} \mathcal{D}_j$$
où $\mathcal{D}_j$ est le cristal de Dieudonné de $X_{\alpha_j}^{D_j - D_{j-1}}$ évalué sur l'épaississement $R \to k$ pour tout $1 \leq j \leq r+1$. Le module $\mathcal{D}$ a une action de $O_L$, et se décompose suivant les éléments de $\T$ en $\mathcal{D} = \oplus_{\tau \in \T} \mathcal{D}_\tau$. De même, on a la décomposition $\mathcal{D}_j = \oplus_{\tau \in \T} \mathcal{D}_{j,\tau}$ pour tout $j$. \\
Soient $\overline{\mathcal{D}_\tau}$ et $\overline{\mathcal{D}_{j,\tau}}$ les réductions de $\mathcal{D}_\tau$ et $\mathcal{D}_{j,\tau}$ modulo $I$ pour tout $\tau \in \T$. Soit $\overline{\Fil_\tau} \subset \overline{\mathcal{D}_\tau}$ la filtration de Hodge. On a 
$$\overline{\Fil_\tau} = \bigoplus_{j=1}^{r+1} \pi^{e-\alpha_{j,\tau}} \overline{\mathcal{D}_{j,\tau}}  $$
Comme le groupe $p$-divisible $G$ est défini sur $R$, on dispose de la filtration de Hodge $\Fil_\tau \subset \mathcal{D}_\tau$ relevant le module $\overline{\Fil_\tau}$. Le $R$-module $\Fil_\tau$ est un facteur direct de $\mathcal{D}_\tau$, pour tout $\tau \in \T$. D'après la théorie de Grothendieck-Messing, prouver l'existence de $G_r \subset G$ revient à prouver que le $R$-module $\Fil_\tau \cap \mathcal{D}_{r+1,\tau}$ est un facteur direct de $\mathcal{D}_{r+1,\tau}$ pour tout $\tau \in \T$. \\
Fixons un élément $\tau \in \T$. Comme $G$ satisfait la condition de Pappas-Rapoport, le module $\Fil_\tau$ contient $\pi^{e-\alpha_{r+1,\tau}} \mathcal{D}_\tau$. En effet, $\alpha_{r+1,\tau}$ est égal au nombre d'entiers $d_{\tau,i}$ ($1 \leq i \leq e$) égaux à $h$. Le module $\Fil_\tau$ est filtré en 
$$0 \subset \Fil_\tau^{[1]} \subset \dots \subset \Fil_\tau^{[e]} = \Fil_\tau$$
Le $R$-module $\Fil_\tau^{[\alpha_{r+1,\tau}]}$ est un facteur direct de $\mathcal{D}_\tau$ de rang $h \alpha_{r+1,\tau}$. De plus, il est contenu dans $\pi^{e-\alpha_{r+1,\tau}} \mathcal{D}_\tau$, qui est également un $R$-module libre de rang $h \alpha_{r+1,\tau}$. D'où l'égalité $\Fil_\tau^{[\alpha_{r+1,\tau}]} = \pi^{e-\alpha_{r+1,\tau}} \mathcal{D}_\tau$. \\
Montrons que $\Fil_\tau \cap \mathcal{D}_{r+1,\tau} = \pi^{e-\alpha_{r+1,\tau}} \mathcal{D}_{r+1,\tau}$. Quitte à quotienter $\mathcal{D}_\tau$ par $\pi^{e-\alpha_{r+1,\tau}} \mathcal{D}_\tau$, on se ramène au cas où $\alpha_{r+1,\tau} = 0$, et nous devons montrer que $\Fil_\tau \cap \mathcal{D}_{r+1,\tau} = 0$. Soit $e_1, \dots, e_l$ une base du $R$-module $\Fil_\tau$, et projetons ces éléments dans $\mathcal{D}_\tau / \mathcal{D}_{r+1,\tau}$. Comme la réduction modulo $I$ de cette famille est libre, elle est également libre sur $R$.  Soit 
$$x = \sum_{k=1}^l x_k e_k \in \Fil_\tau \cap \mathcal{D}_{r+1,\tau}$$
En projetant $x$ sur $\mathcal{D}_\tau / \mathcal{D}_{r+1,\tau}$, on voit que $x_1 = \dots = x_l = 0$, soit $x=0$. \\
Nous avons prouvé que pour tout $\tau \in \T$, $\Fil_\tau \cap \mathcal{D}_{r+1,\tau} = \pi^{e-\alpha_{r+1,\tau}} \mathcal{D}_\tau$. Cela prouve l'existence (et l'unicité) du groupe $p$-divisible $G_r$. En raisonnant par récurrence, en considérant $G_r$ au lieu de $G$, on en déduit l'existence et l'unicité de tous les $G_j$, $1 \leq j \leq r$. Nous avons également démontré que chacun des $X_{\alpha_j}^{D_j - D_{j-1}}$ était déformé trivialement, i.e. $G_j / G_{j-1} \simeq X_{\alpha_j}^{D_j - D_{j-1}} \times_k R$ pour $1 \leq j \leq r+1$. \\
Décrivons maintenant comment prouver le cas général, lorsque $I$ n'est plus supposé de carré nul. On peut trouver un entier $N \geq 1$, et des anneaux $R_1, \dots, R_N$ tels $R_1 = k$, $R_N = R$, et pour $1 \leq n \leq N-1$, il existe un morphisme surjectif $R_{n+1} \to R_n$, dont le noyau est de carré nul. En utilisant le même raisonnement que précédemment, on montre alors par récurrence sur $n$ qu'il existe un unique relèvement des groupes $p$-divisibles $G_j$ à $R_n$, tels que $G_j / G_{j-1}$ est déformé trivialement pour $1 \leq j \leq r+1$. Cela termine la preuve.
\end{proof}

\begin{rema}
La proposition précédente est fausse si on ne suppose pas que $G$ satisfait la condition de Pappas-Rapoport sur $R$. En effet, supposons que $L$ est totalement ramifié de degré $e=3$, et considérons la donnée $h=2$, $(d_1, d_2, d_3)=(0,1,2)$. Le groupe $p$-divisible $\mu$-ordinaire sur $k = \overline \FP$ pour cette donnée est $G = X_1 \times X_2$, o\`u  les polygones de Newton de $X_1$ et $X_2$ ont une pente respectivement égale \`a $2/3$ et $1/3$.
Le cristal de Dieudonné de $G$ évalué sur $k$ est un $k \otimes_{\mathbb{Z}_p} O_L$-module libre de rang $2$. Soit $e_1, e_2$ la base naturelle de ce module correspondant \`a la décomposition $X_1 \times X_2$. Une base de la filtration de Hodge sur $k$ est alors $\overline{\Fil} = (\pi e_1,\pi^2 e_1,\pi^2 e_2)$. \\
Soit $k[\varepsilon]$ les nombres duaux. On peut alors considérer le relèvement admissible et stable par $O_L$ de la filtration de Hodge
\[ \Fil = (\pi e_1 + x \varepsilon \pi e_2 ,\pi^2 e_1,\pi^2 e_2+  y \varepsilon e_1), \quad x,y \dans k\]
Il correspond à un groupe $p$-divisible avec action de $O_L$ sur $k[\varepsilon]$, mais celui-ci n'a pas de filtration canonique si $x$ et $y$ sont non nuls. De plus, on vérifie facilement que ce groupe $p$-divisible ne satisfait pas de condition PR pour la donnée étudiée.
\end{rema}

\begin{rema}
Si les entiers $d_{\tau,i}$ sont tous égaux à $0$ ou $h$, alors $r=0$. La preuve de la proposition précédente implique alors que les déformations du groupe $p$-divisible $X_{\alpha_1}^h$ satisfaisant la condition de Pappas-Rapoport sont triviales.
\end{rema}

\begin{rema}
Si le corps $k$ n'est pas algébriquement clos, le groupe $G \times_R k$ n'est plus nécessairement isomorphe à $X^{ord}$, mais on a $G \times_R k \simeq \prod_{j=1}^{r+1} X_j$, où $X_j$ est un groupe $p$-divisible sur $k$ avec $X_j \times_k \overline{k} \simeq X_{\alpha_j}^{D_j - D_{j-1}}$ pour $1 \leq j \leq r+1$. Un argument de descente permet alors de montrer qu'il existe une filtration 
$$0=G_0 \subset G_1 \subset \dots \subset G_r \subset G_{r+1}=G  $$
telle que $ (G_j / G_{j-1}) \times_{R} k \simeq X_j$ pour $1 \leq j \leq r+1$. 
\end{rema}

 \nocite{*}
\bibliographystyle{alpha-fr} 

\bibliography{biblio} 

\end{document}